\documentclass[12pt,reqno]{amsart}
\usepackage[utf8]{inputenc}
\usepackage{a4wide}
\numberwithin{equation}{section}
\usepackage{mathrsfs}
\usepackage{stmaryrd}
\usepackage{url}
\usepackage{amsfonts}
\usepackage{amscd}
\usepackage{indentfirst}
\usepackage{enumerate}
\usepackage{amsmath,amsfonts,amssymb,amsthm,mathtools}
\usepackage{graphicx,mathrsfs}
\usepackage{appendix}
\usepackage{enumitem}
\usepackage[numbers,sort&compress]{natbib}
\usepackage{color}
 \usepackage[colorlinks, linkcolor=blue, citecolor=blue]{hyperref}
\usepackage{float}

\newcommand{\dist}{\operatorname{dist}}

\newcommand\R{\mathbb R}

\newcommand\mbb\mathbb
\newcommand\mbf\mathbf
\newcommand\mcal\mathcal
\newcommand\mfrak\mathfrak
\newcommand\mrm\mathrm
\newcommand\msf\mathsf
\renewcommand\a\alpha
\renewcommand\b\beta
\newcommand\g\gamma
\newcommand\G\Gamma
\renewcommand\d\delta
\newcommand\D\Delta
\newcommand\e\epsilon
\newcommand\z\zeta
\renewcommand\t\theta
\newcommand\Th\Theta
\newcommand\la\lambda
\newcommand\La\Lambda
\newcommand\s\sigma
\newcommand\si\varsigma
\newcommand\Si\Sigma
\newcommand\ups\upsilon
\newcommand\U\Upsilon
\newcommand\ph\varphi
\renewcommand\o\omega
\renewcommand\O\Omega
\newcommand\wt\widetilde
\newcommand\wh\widehat
\newcommand\ol\overline
\newcommand\ul\underline
\newcommand\mr\mathring
\newcommand\ub\underbrace
\newcommand\pa\partial
\newcommand\n\nabla
\newcommand\fa\forall
\newcommand\ex\exists
\newcommand\es\emptyset
\newcommand\wk\rightharpoonup
\newcommand\inc\hookrightarrow
\newcommand\linf\varliminf
\newcommand\lsup\varlimsup
\newcommand\os\overset
\newcommand\us\underset
\newcommand\sr\stackrel
\newcommand\Ot\Leftarrow
\newcommand\To\Rightarrow
\newcommand\map\mapsto
\newcommand\ot\leftarrow
\newcommand\lot\longleftarrow
\newcommand\lto\longrightarrow
\newcommand\tot\leftrightarrow
\newcommand\ltot\longleftrightarrow
\newcommand\sm\backslash
\renewcommand\Cup\bigcup
\renewcommand\Cap\bigcap
\newcommand\sub\subset
\newcommand\Sub\Subset
\newcommand\sne\subsetneq
\newcommand\bus\supset
\newcommand\Bus\Supset

\newcommand\eq\equiv
\newcommand\ox\otimes
\newcommand\Ox\bigotimes
\newcommand\pl\oplus
\newcommand\Pl\bigoplus
\newcommand\x\times
\renewcommand\c\circ
\newcommand\q\quad
\newcommand\fr\frac

\usepackage{graphicx}
\usepackage{subcaption}

\usepackage[numbers,sort&compress]{natbib}
\usepackage[dvipsnames]{xcolor}
\usepackage{color}

\definecolor{bondiblue}{rgb}{0.0, 0.58, 0.71}
\def\sideremark#1{\ifvmode\leavevmode\fi\vadjust{\vbox to0pt{\vss% the remark
			\hbox to 0pt{\hskip\hsize\hskip1em%                          will appear only
				\vbox{\hsize2.1cm\tiny\raggedright\pretolerance10000%          on the side
					\noindent #1\hfill}\hss}\vbox to15pt{\vfil}\vss}}}%

\newtheorem{Thm}{Theorem}[section]
\newtheorem{Lem}[Thm]{Lemma}
\newtheorem{Cor}[Thm]{Corollary}
\newtheorem{Prop}[Thm]{Proposition}

\newtheorem{Def}[Thm]{Definition}
\newtheorem{Rem}[Thm]{Remark}

\newcommand{\eps}{\epsilon}
\newcommand{\dd}{\mathop{}\!\mathrm{d}}
\newcommand{\lammax}{\lambda_{\max}}

\DeclareMathOperator{\supp}{supp}

\title[Hessian degeneracy of the torsion function]
{Hessian degeneracy of the torsion function on smooth simply connected nonconvex planar domains}

\author{Xiuda Liang}
\address[Xiuda Liang]{School of Mathematics and Statistics, Central China Normal University, Wuhan 430079, China}
\email{lxddd@mails.ccnu.edu.cn}

\author{Peng Luo}
\address[Peng Luo]{School of Mathematics and Statistics, Key Laboratory of Nonlinear Analysis and Applications (Ministry of Education), and Hubei Key Laboratory of Mathematical Sciences, Central China Normal University, Wuhan 430079, China}
\email{pluo@ccnu.edu.cn}

\author{Wenjie Wang}
\address[Wenjie Wang]{School of Mathematics and Statistics, Central China Normal University, Wuhan 430079, China}
\email{wjwang3269@mails.ccnu.edu.cn}

\begin{document}

\begin{abstract}
We construct a family of bounded, smooth, simply connected, nonconvex planar domains $\Omega_{a,\epsilon}$. Let $u_{a,\epsilon}$ be  the corresponding torsion function satisfying
\[
        -\Delta u_{a,\epsilon}=1\quad\text{in }\Omega_{a,\epsilon},\qquad u_{a,\epsilon}=0\quad\text{on }\partial\Omega_{a,\epsilon}.
\]  There exists \(a^*\in(0,1)\) such that, for every \(a\in(a^*,1)\) and all sufficiently small \(\epsilon>0\), the origin is a unique global maximizer of \(u_{a,\epsilon}\). Moreover, 
\[ \lim_{a\downarrow a^*}\lim_{\epsilon\to0} \lambda_{\max}\bigl(D^2u_{a,\epsilon}(0)\bigr)=0. \]
In addition, the ratios $\text{diam}(\Omega_{a,\epsilon}) / \text{inrad}(\Omega_{a,\epsilon})$ are uniformly bounded.  Hence the Hessian estimate proved by Steinerberger (J. Funct. Anal. 274, 1611--1630, 2018) for convex planar domains cannot be extended to  smooth, simply connected, nonconvex planar domains.
\vskip0.2cm
 Our domains are star-shaped, symmetric with respect to both coordinate axes and convex in the horizontal direction. When the limiting slit is sufficiently long,  we  further show that certain superlevel sets of the torsion function are not star-shaped. This strengthens the counterexample to the star-shapedness question raised by Gladiali and Grossi (Amer. J. Math. 144, 1221--1240, 2022) under stronger geometric 
assumptions. 
\end{abstract}

\maketitle
{\small

\section{Introduction and main results}
\setcounter{equation}{0}
In this paper, we study the classical torsion problem
 \begin{equation}\label{1h}
\begin{cases}
-\Delta u=1~&\mbox{in}\ \Omega,\\[1mm]
\quad \; \;u=0~&\mbox{on}\ \partial \Omega.
\end{cases}
\end{equation}
 The torsion problem \eqref{1h} originates in Saint-Venant's theory of elastic torsion and has become a classical model in elliptic PDE and geometric analysis, particularly in the study of maximum points, concavity, and the geometry of level sets.
Several qualitative properties of the torsion function on convex planar domains are well understood. A classical theorem of Makar-Limanov \cite{ML71} shows that  \(\sqrt{u}\) is concave on planar convex domains, which implies the convexity of its superlevel sets. Let \(y_0\) denote the unique maximizer and let \(\lambda_1, \lambda_2\) be the eigenvalues of the Hessian \(D^2u(y_0)\). Since \(-\Delta u=1\), we have
\[
\lambda_1+\lambda_2=\operatorname{tr} D^2u(y_0)=\Delta u(y_0)=-1.
\]
The eccentricity of the level sets near \(y_0\) is governed by these eigenvalues. If one eigenvalue is close to zero, the level sets become highly flattened and eccentric. However, Steinerberger \cite{Ste18} showed that such extreme eccentricity cannot occur when $\O$ is convex. In general, the behavior of critical points of torsion functions in nonconvex domains has attracted considerable attention. In particular, Gladiali and Grossi  \cite{GG2022} constructed a family of star-shaped simply connected domains on which the torsion function possesses an arbitrarily large number of maximum points. For further properties of the torsion function, see \cite{KM1993,HNST,K1987,B2020,HS2021,HNS2016}.
 \vskip 0.2cm
 
More generally,  related questions arise for the semilinear Dirichlet problem
 \begin{equation}\label{f(u)}
\begin{cases}
-\Delta u=f(u)~&\mbox{in}\ \Omega,\\[1mm]
\quad \; \; u=0~&\mbox{on}\ \partial \Omega.
\end{cases}
\end{equation}
 A well-known theorem of Gidas, Ni and Nirenberg in
 \cite{GNN1979}, proved by the moving-plane method,  implies that the superlevel sets of the solution $u$ of \eqref{f(u)} are convex, i.e., $u$ is quasiconcave, when $\O$ is a disk. Lions \cite{lions} conjectured in 1981 that  $u$ should be quasiconcave for general  $f$, and 
the question of quasiconcavity for \eqref{f(u)} has received renewed interest.  Recently, Steinerberger \cite{Steinerberger} established that under certain conditions on $f$, the concavity of $u$ propagates from the boundary to the interior. Chau and Weinkove \cite{weinkove2023} subsequently obtained related concavity results under weaker assumptions on $f$ and provided a geometric criterion involving ellipses.
 Related results can be found in \cite{BL1976,GG2004-1,GG2004-2,K1985, kennington1985power,CF1985,CS1982,CC1998,Grossi2024}.

\vskip 0.2cm

  Our starting point is Steinerberger's Hessian estimate
 in \cite{Ste18}, which we recall next.

\vskip 0.2cm

\noindent \textbf{Theorem A (Steinerberger \cite{Ste18}).} \emph{Let $\Omega\subset \R^2$ be a bounded  convex domain and $u_0(x)$ be the solution of problem \eqref{1h} with  maximizer $y_0\in \Omega$. Then there exist universal constants $c_1, c_2>0$ such that}
\begin{equation*}
\lambda_{\max}\left(D^2u_0(y_0)\right)\leq -c_1\mbox{exp}\left(-c_2\frac{\text{diam}(\Omega)}{\mbox{inrad}(\Omega)}\right).
\end{equation*}

\vskip 0.2cm

In \cite{Ste18}, the proof of Theorem A relies essentially on the convexity of the domain. Steinerberger therefore  proposed the following open problem: \textbf{\textit{``Does the above result hold true on domains that are not convex but merely simply connected or perhaps only bounded? The proof uses convexity of the domain $\Omega$ in a very essential way and it is not clear to us whether the statement remains valid in other settings."}}

 \vskip 0.2cm
Chen and Luo  \cite{C-L} constructed counterexamples by removing a small interior disk from a convex domain. Let $\Omega_\epsilon=\O\backslash B(x_0,\epsilon)$, where $\Omega\subset \R^2$ is  bounded and convex, $x_0\in\O$ and $B(x_0,\epsilon)$ denotes the ball centered at $x_0$ with radius $\epsilon$. Suppose $y_0$ is the maximizer of $u_0(x)$ as in Theorem A, $u_\e$ is the torsion function of $\Omega_\epsilon$, and $x_{\epsilon}\in\Omega_{\epsilon}$ is the maximizer of $u_{\epsilon}$. If $\lambda_1$ and $\lambda_2$ are the eigenvalues of $D^2u_0(y_0)$, then
\begin{equation*}
 \lim_{\e\to 0}\lambda_{\max}\big(D^2u_\e(x_\e)\big)=
 \begin{cases}
 \max\big\{\lambda_1,\lambda_2\big\} &\mbox{if} ~x_0\neq y_0,\\[2mm]
 \max\big\{\lambda_1,\lambda_2,-|\lambda_2-\lambda_1|\big\} &\mbox{if}~x_0= y_0.
 \end{cases}
\end{equation*}
Their construction shows that convexity is essential. \textbf{However, it destroys both convexity and simple connectivity by introducing an interior boundary component}. In their construction, the degeneration of the Hessian spectral gap is driven by the topological change.  \textbf{Thus their construction does not address the simply connected alternative explicitly raised in Steinerberger's question}. To the best of our knowledge, this case has
remained open. By contrast, the domains we construct are smooth, simply connected and nonconvex, and the Hessian degeneracy is produced by
the geometry of the outer boundary.
\begin{figure}[H]
    \centering
    \begin{subfigure}[b]{0.45\textwidth}
        \centering
        \includegraphics[width=0.6\linewidth]{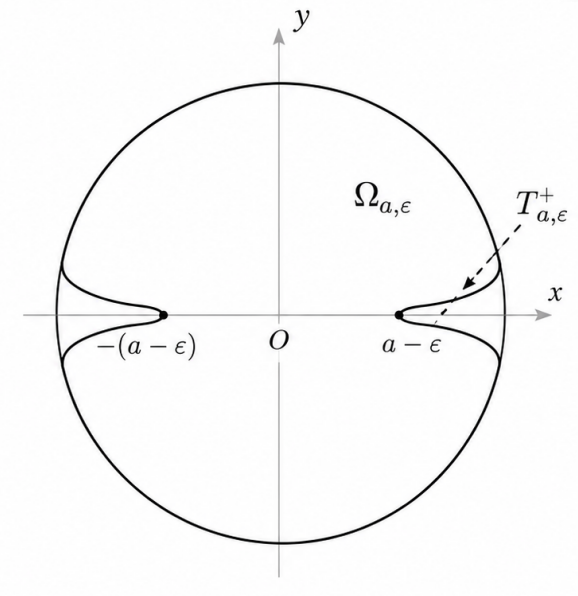}
        \caption{$\Omega= \Omega_{a,\epsilon}$}
        \label{fig:sub1}
    \end{subfigure}
    \hfill
    \begin{subfigure}[b]{0.45\textwidth}
        \centering
        \includegraphics[width=0.6\linewidth]{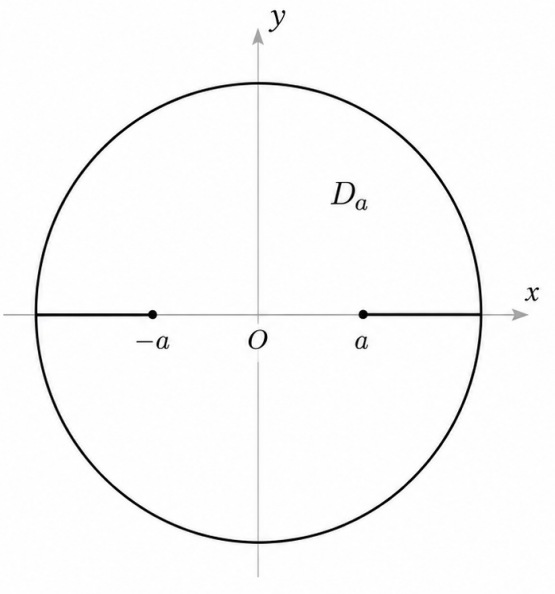}
        \caption{$\O = D_a$} 
        \label{fig:sub2}
    \end{subfigure}
\end{figure}

 We now introduce the family of domains and consider the corresponding torsion problem
\begin{equation}\label{fc1.1}
\begin{cases}
-\Delta u=1  &{\text{in}~\Omega_{a,\epsilon}},\\[0.5mm]
\quad \; \; u=0 &{\text{on}~\partial \Omega_{a,\epsilon}},
\end{cases}
\end{equation}
where $a\in (0,1)$,   
$\Omega_{a,\epsilon}=
B_1(0)\setminus \overline{T_{a,\epsilon}}$ and $T_{a,\epsilon}$ denotes the open slot, with
$\overline{T_{a,\epsilon}}
=T_{a,\epsilon}\cup\partial T_{a,\epsilon}$
(see Figure~\subref{fig:sub1}). We further assume that \begin{equation}\label{quyuomiga}
    \displaystyle \Omega_{a,\eps}
=
\big\{
(x,y)\in B_1(0):
|y|<1,~
-\rho_{a,\eps}(y)<x<\rho_{a,\eps}(y)
\big\},
\end{equation}
where $\displaystyle \rho_{a,\epsilon}
\in C^\infty\big((-1,1)\big)\cap C\big([-1,1]\big)$ is even and satisfies $\displaystyle 0<\rho_{a,\epsilon}(y)\leq\sqrt{1-y^2}$ for $|y|<1$, and $\rho_{a,\epsilon}(\pm1)=0$.  We assume in addition that the family $\Omega_{a,\eps}$ satisfies the following properties: 
\begin{enumerate}[label=(P\arabic*), ref=(P\arabic*)]
    \item \label{C1}
    \(\Omega_{a,\epsilon}\) is a bounded, \(C^\infty\),
    simply connected, nonconvex planar domain and is symmetric with respect to both coordinate axes;

    \item \label{C2}
    \(\Omega_{a,\epsilon}\subset D_a\), where 
    \(D_a := B_1(0)\setminus {S_a}\) and 
    \(S_a:=\bigl((-1,-a]\cup[a,1)\bigr)\times\{0\}\) (see Figure \subref{fig:sub2});
 \item \label{C3}
    In the Hausdorff distance,  
    $  d_{{H}}\!\left(\overline{{T_{a,\epsilon}}},\overline{S_a} \right) \to 0
    ~ \text{as }\epsilon\to0$, where $\overline{S_a}
=
\bigl([-1,-a]\cup[a,1]\bigr)\times\{0\}$.
 \item \label{C4}
    There exists $\e_0(a)>0$ such that, whenever \(0<\epsilon_1<\epsilon_2<\epsilon_0(a)\), 
    one has \(\Omega_{a,\epsilon_2}\subset \Omega_{a,\epsilon_1}\), and $\displaystyle \bigcup_{0<\epsilon<\epsilon_0(a)} \Omega_{a,\epsilon} = D_a$. 
    \item \label{C5}$\Omega_{a,\epsilon}$ is star-shaped with respect
to the origin, that is,
$
 t\Omega_{a,\epsilon}\subset\Omega_{a,\epsilon}
 ~\text{for every }t\in[0,1].
$
\item \label{C6} $\Omega_{a,\epsilon}$ is convex in the $x$-direction, namely,
for any $y\in(-1,1)$, if $(x_1,y),(x_2,y)\in\Omega_{a,\epsilon}$,
then $\bigl((1-t)x_1+tx_2,y\bigr)\in\Omega_{a,\epsilon}$ for any $t\in[0,1]$. 
\end{enumerate}
 Observe that \ref{C6} follows immediately from \eqref{quyuomiga}. Explicit examples satisfying \eqref{quyuomiga}  and \ref{C1}--\ref{C6} are constructed in Appendix \ref{A}. The next theorem implies that the estimate in Theorem A does not 
extend, with universal constants, to 
smooth simply connected 
nonconvex domains.

\begin{Thm}\label{Thm1.1}
Let $u_{a,\epsilon}$ be the solution of \eqref{fc1.1}.
There exists $a^*\in(0,1)$ such that, for every $a\in(a^*,1)$,
we can find $\epsilon(a)>0$ such that, for every
$0<\epsilon<\epsilon(a)$, the origin is a unique global maximizer
of $u_{a,\epsilon}$ in $\Omega_{a,\epsilon}$. Define
\[
\Lambda_{\epsilon}(a)=\lambda_{\max}\bigl(D^{2}u_{a,\epsilon}(0)\bigr).
\]
Then
\[
\lim_{a\downarrow a^{*}}\lim_{\epsilon\to0}\Lambda_{\epsilon}(a)=0.
\]
Moreover, the ratios $\text{diam}(\Omega_{a,\epsilon}) / \text{inrad}(\Omega_{a,\epsilon})$ are uniformly bounded. More precisely,
\begin{equation}\label{inrad}
    \sup\limits_{\substack{a\in(0,1)\\
0<\epsilon<\epsilon_0(a)}}
\displaystyle\frac{\operatorname{diam}(\Omega_{a,\epsilon})}
{\operatorname{inrad}(\Omega_{a,\epsilon})}
\leq 8.
\end{equation}

\end{Thm}
\begin{Rem}
Theorem~\ref{Thm1.1} gives a negative answer to Steinerberger's
question for smooth, simply connected nonconvex domains. If
Steinerberger's estimate held for this class of domains with universal
constants \(c_1,c_2>0\), then
\[
\Lambda_\epsilon(a)
\leq
-c_1\exp\left(
-c_2\frac{\operatorname{diam}(\Omega_{a,\epsilon})}
{\operatorname{inrad}(\Omega_{a,\epsilon})}
\right)
\leq
-c_1e^{-8c_2}<0.
\]
It would then follow that $\lim\limits_{a\downarrow a^*}\lim\limits_{\epsilon\to0}
\Lambda_\epsilon(a)
\leq -c_1e^{-8c_2}<0$, contradicting
$
\lim\limits_{a\downarrow a^*}\lim\limits_{\epsilon\to0}
\Lambda_\epsilon(a)=0$.
Hence Steinerberger's estimate cannot be extended to smooth, simply connected nonconvex domains.
\end{Rem}
 Our construction also relates to a question raised by Gladiali and Grossi in \cite{GG2022}: \textbf{\textit{“If 
$\Omega$ is a star-shaped domain, does it follow that the superlevel sets of $u$ are also star-shaped?"}} This question is motivated by  Gidas, Ni and Nirenberg \cite{GNN1979}. They showed  that, if $f$ is  locally Lipschitz and $\O $ is  bounded, smooth, symmetric with respect to each coordinate plane, and convex in each coordinate direction, then every superlevel set of the positive solution to \eqref{f(u)} is star-shaped with respect to the origin.  Gladiali and Grossi \cite{GG2022} therefore posed this question and constructed a counterexample in which the domain is star-shaped but lacks the symmetry and convexity assumptions, and the superlevel sets of the  torsion function are not star-shaped. In contrast, our domains are symmetric with respect to both coordinate axes and convex in the x-direction.
The next result shows that these stronger geometric assumptions still do not guarantee
star-shaped superlevel sets.

\begin{Thm}\label{Thm1.2}
    Let $u_{a,\epsilon}$ be the solution of \eqref{fc1.1}. Then there exists  ${b}^{*}  \in (0,1)$ such that for every $a\in(0,b^*)$, we can find $\delta (a)>0$ such that, for any $\displaystyle \epsilon\in \big(0,\delta (a)\big)$,
\(u_{a,\epsilon}\) has at least two distinct global maximizers in \(\Omega_{a,\epsilon}\).
\end{Thm}
 \begin{Cor}\label{cor1.3}
Let $u_{a,\epsilon}$ be the solution of \eqref{fc1.1}. 
Fix $a\in(0,b^*)$, where $b^*$ is as in Theorem \ref{Thm1.2}. 
Then there exists $\epsilon^{*}(a)>0$ such that, for all $\displaystyle \epsilon\in\big(0,\epsilon^{*}(a)\big)$, 
the domain $\Omega_{a,\epsilon}$ is smooth, symmetric with respect to both coordinate axes, star-shaped with respect to the origin, and convex in the $x$-direction, whereas the superlevel set
\begin{equation}
\label{1.7}
\mathcal{L}_{a,\epsilon}
:=
\big\{
(x,y)\in\Omega_{a,\epsilon}:
u_{a,\epsilon}(x,y) > u_{a,\epsilon}(0,0)
\big\}
\end{equation}
is not star-shaped. In fact, $\mathcal{L}_{a,\epsilon}$ is disconnected.
\end{Cor}

The principal difficulty is to produce Hessian degeneracy through the outer-boundary geometry while preserving smoothness and simple connectivity. This requires a precise analysis of the singular slit limit, including the asymptotic behavior of the Hessian, as well as a global maximum argument in the absence of full convexity.

\vskip0.2cm

To overcome these challenges, we first analyze the torsion function $u_a$ on the slit domain $D_a$, and then transfer these conclusions to the smooth approximating domains using
Mosco convergence and capacity. The variational convergence argument under domain perturbation is motivated  by \cite{BV2000}. To analyze the Hessian of the torsion function $u_a$ at the origin as $a\rightarrow0$ and 
$a \rightarrow 1$, we establish boundary regularity on the slit domain $D_a$ and construct an explicit comparison function to derive sharp asymptotics for the Hessian as $a \rightarrow 0$. These results enable us to identify the critical parameter $a^*$ such that the largest Hessian eigenvalue tends to zero as $a\to a^*$. In particular, Lemma \ref{lem:horizontal-monotonicity} uses a reflection argument to show that every global maximizer of the torsion function lies on the \(y\)-axis, which is inspired by the moving-plane method in \cite{GNN1979}. Furthermore, Lemma \ref{green-monotonicity} establishes the strict monotonicity of  $u_a$ along the \(y\)-axis using the Green function and Radon measure representation. Lemma \ref{lem:horizontal-monotonicity} and Lemma \ref{green-monotonicity}  identify the origin as the unique global maximizer in the relevant parameter regime, forming the core of the proof of Theorem \ref{Thm1.1}. 
 \vskip 0.2cm
The paper is organized as follows.  In Section \ref{preliminary results}, we introduce some  basic analytical results that will be used throughout the paper. In Section \ref{sign change}, we establish the change of Hessian signature and use it to define
the critical parameters $a^*$ and $b^*$. Moreover, we  prove Theorem \ref{Thm1.2} and Corollary \ref{cor1.3}.  Section \ref{global maximum} is devoted to proving Theorem \ref{Thm1.1}. Finally, some  technical proofs are deferred to the appendices.
\vskip 0.2cm
Throughout the paper, we set
$B_r:=B_r(0)$,  $B_r^+:=B_r(0)\cap\{y>0\}$,
and write $B:=B_1$,  $B^+:=B_1^+$. Unless otherwise specified,
every function in $H_0^1(G)$, $G\subset B$, is identified with its
zero extension to $B$. In particular, $u_a$ and
$u_{a,\epsilon}$ are regarded as elements of $H_0^1(B)$.

\section{Preliminaries and convergence to the slit domain}\label{preliminary results}
In this section, we recall some basic facts about capacity and Mosco convergence and establish several useful results on symmetry and regularity.

\vskip0.1cm

 Let \(u_a\in H^1_0(D_a)\) be the weak solution of $- \Delta u=1$ in $D_a$ and define \begin{equation}\label{Ju}
    J(u)=\frac12\int_{B}|\nabla u|^2\dd x-\int_{B}u\dd x.
\end{equation}
The existence and uniqueness of $u_a$ follow from the Lax--Milgram theorem.  We next  improve the boundary regularity of  $u_a$ by a capacity method, and the proof is postponed to Appendix \ref{B}.
\begin{Lem}\label{capacity}
Every point of $\partial D_a$ is regular for the Dirichlet problem. Consequently,
\[
u_a \in C(\overline{D_a}), \quad u_a = 0 \quad \text{pointwise on } \partial D_a.
\]
\end{Lem}
\begin{proof}
In a planar slit domain, boundary regularity of the Perron solution is established through the Wiener criterion using a capacity estimate. The full proof is given in Appendix~\ref{B}, and a similar argument can be found in \cite{BV2000}.
\end{proof}

\vskip 0.1cm
\begin{Lem}\label{Lem2.2}
Let $u_{a, \e}$ be the solution of \eqref{fc1.1}, and $u_a$ be the solution of \eqref{1h}  when  $\O = D_a$. Since $\Omega_{a,\epsilon} \subset D_a$, we have
\begin{equation}\label{u_a}
0 < u_{a,\e} \le u_a \quad \text{in } \Omega_{a,\epsilon}.
\end{equation}
\end{Lem}
\begin{proof}
Set $v := u_a - u_{a,\epsilon}$ in $\Omega_{a,\epsilon}$. Then 
$-\Delta v = 0 ~ \text{in } \Omega_{a,\epsilon}$. Since $u_{a,\e}=0$ on $\partial\Omega_{a,\e}$, this yields 
\[
v=u_{a}\ge 0
\quad \text{on } \partial\Omega_{a,\e} .
\]
The strong maximum principle gives $u_{a, \e}>0$ in $\Omega_{a,\epsilon}$, while the maximum principle applied to $v$
yields $v\geq 0$ in $\Omega_{a,\epsilon}$, proving \eqref{u_a}.
\end{proof}
 \vskip 0.1cm
 
 \subsection{Results on symmetry and horizontal monotonicity}
\begin{Lem}
For every $a\in(0,1)$, the torsion function $u_a$ is positive in $D_a$
and  even in both variables:
\[
u_a(x,y)=u_a(-x,y)=u_a(x,-y).
\]
In particular,
\begin{equation}\label{mu(a)}
\nabla u_a(0)=0,
\qquad
D^2u_a(0)=
\begin{pmatrix}
-1-\mu(a) & 0\\
0 & \mu(a)
\end{pmatrix},
\qquad
\mu(a):=\partial_{yy}u_a(0).
\end{equation}
Consequently,
\begin{equation}\label{Lambda-formula}
\Lambda(a):=\lammax\bigl(D^2u_a(0)\bigr)
=\max\bigl\{\mu(a),-1-\mu(a)\bigr\}.
\end{equation}
\end{Lem}

\begin{proof}
By Lemma \ref{capacity} and the maximum principle, $u_a>0$ in $D_a$.
The symmetry of $D_a$ and the uniqueness of the Dirichlet problem imply that
$u_a$ is even in both variables. Since $u_a$ is smooth near the origin,
$\nabla u_a(0)=0$ and  $\partial_{xy}u_a(0)=0$.
Since $-\Delta u_a=1$ near the origin,
\[
\partial_{xx}u_a(0)+\partial_{yy}u_a(0)=-1,
\]
which gives \eqref{mu(a)} and \eqref{Lambda-formula}.
\end{proof}

 We next use a planar reflection argument to show that every maximum point of $u_a$ and $u_{a,\e}$ lies on the \(y\)-axis, which relies crucially  on the symmetry of $D_a$ and $\O_{a,\e}$.

\begin{Lem}\label{lem:horizontal-monotonicity}
Let \(G\) be either  \(\Omega_{a,\epsilon}\) or \(D_a\), and let
\(u\in C^2(G)\cap C(\overline{G})\) be the corresponding torsion
function, namely,
\begin{equation*}
-\Delta u=1 \quad \text{in }G,
\quad  u=0  \quad\text{on }\partial G.
\end{equation*}
Then 
\begin{equation*}
u(0,y)>u(x,y)
\quad
\text{for every }(x,y)\in G\text{ with }x\neq0.
\end{equation*}
Consequently, every global maximizer of $u$ lies on the $y$-axis.
\end{Lem}
\begin{proof}
Since \(G\) is symmetric with respect to the \(y\)-axis, the uniqueness
of the torsion function gives $u(-x,y)=u(x,y)$ for any $(x,y)\in G$.
By \eqref{quyuomiga} and the definition of $D_a$, for each \(y \in (-1,1)\), there exists
\(\rho_G(y)>0\) such that
\[
\big\{x\in\mathbb{R}:(x,y)\in G\big\}
=
\bigl(-\rho_G(y),\rho_G(y)\bigr).
\]
Fix \(\lambda>0\) and define $$\Sigma_\lambda:=G\cap\{x>\lambda\} \text{ and } R_\lambda(x,y):=(2\lambda-x,y).$$
 We claim that
$R_\lambda(\Sigma_\lambda)\subset G$.
Indeed, if \((x,y)\in\Sigma_\lambda\), then $\lambda<x<\rho_G(y)$.
Since \(\lambda>0\), we have $-\rho_G(y)
<
2\lambda-x
<
\rho_G(y)$,
which proves $R_\lambda(\Sigma_\lambda)\subset G$.
Define
$$w_\lambda(x,y)
:=
u\bigl(R_\lambda(x,y)\bigr)-u(x,y), \quad \text{for }(x,y)\in\Sigma_\lambda.$$
Then $\Delta w_\lambda=0$ in $\Sigma_\lambda$.
Moreover, $w_\lambda=0$ on
$\partial\Sigma_\lambda\cap\{x=\lambda\}$, while
$w_\lambda\geq0$ on
$\partial\Sigma_\lambda\cap\partial G$.
Hence, by the maximum principle,
\begin{equation}
\label{eq:w-nonnegative}
w_{\lambda}\geq0
\quad \text{in }\Sigma_{\lambda}.
\end{equation}
 Let \(\mathcal{C}\) be any nonempty
connected component of \(\Sigma_\lambda\), and choose
\((x_0,y_0)\in \mathcal{C}\). Then
$\lambda<x_0<\rho_G(y_0)$,
and the  segment
$\bigl\{(x,y_0):x_0<x<\rho_G(y_0)\bigr\}$
is contained in \(\mathcal{C}\). Therefore,
$$p:=\bigl(\rho_G(y_0),y_0\bigr)
\in\partial \mathcal{C}\cap\partial G.$$
Moreover,
$R_\lambda(p)
=
\bigl(2\lambda-\rho_G(y_0),y_0\bigr)$.
Since \(\lambda<\rho_G(y_0)\), we have
$-\rho_G(y_0)
<
2\lambda-\rho_G(y_0)
<
\rho_G(y_0)$,
and hence
$R_\lambda(p)\in G$.
By the strong maximum principle, we obtain
$u\bigl(R_\lambda(p)\bigr)>0$.
Since \(u(p)=0\),
$$w_\lambda(p)
=
u\bigl(R_\lambda(p)\bigr)-u(p)>0.$$
Thus \(w_\lambda\) is not identically zero in \(\mathcal{C}\). Combining this with
\eqref{eq:w-nonnegative} and the strong maximum principle, we obtain $w_\lambda>0$ in $\mathcal{C}$.
Since \(\mathcal{C}\) was arbitrary, we have $$w_\lambda>0 \quad\text{in } \Sigma_\lambda.$$
 Let \((x,y)\in G\) with \(x>0\) and choose
$\lambda={x}/{2}$.
Then \((x,y)\in\Sigma_\lambda\) and
$R_\lambda(x,y)=(0,y)$.
Therefore, 
\[
u(0,y)-u(x,y)
=
w_\lambda(x,y)>0.
\]
The case \(x<0\) follows from the symmetry of \(u\) with respect to the
\(y\)-axis.
\end{proof}

\subsection{Mosco convergence }$\,$  \vskip 0.2cm  We recall the definition of Mosco convergence, and see \cite{Attouch1984,BV2000} for  further properties.
    \begin{Def}
 Let $E, ~E_n \subset B$. 
  We say that $H_0^1(E_n)$ converges to $H_0^1(E)$ in the  sense of Mosco  if the following two conditions hold:
  \begin{enumerate}[label=(M\arabic*), ref=(M\arabic*)]
    \item \label{M1}(Mosco-liminf condition)  If $u_n \in H_0^1(E_n)$ and ${u}_n\rightharpoonup{u}
~\text{weakly in }H_0^1(B)$,
 then $u \in H_0^1(E)$.
    \item \label{M2}(Mosco-limsup condition) For any given $u \in H_0^1(E)$, there exists  $u_n \in H_0^1(E_n)$ such that \[{u}_n \to {u}
\quad \text{strongly in } H_0^1(B).\]    
  \end{enumerate}
  \end{Def}

\begin{Lem}\label{Mosco2}
For every $a_0 \in [0,1]$, we have
\begin{equation}\label{Mosco1-1}
H_0^1(D_{a})
\xrightarrow{M}
H_0^1(D_{a_0})
\quad \text{as}~~  a \to a_0.
\end{equation}
\end{Lem}
\begin{proof}
Define $$S_0:= (-1,1)\times \{0\}, ~D_0=B \setminus S_0, ~ \text{and} ~ D_1=B.$$
As usual, we identify $H_0^1(D_a)$ with a closed linear subspace of
$H_0^1(B)$ by zero extension. Let $\{a_n\}\subset [0,1]$ be any sequence such that $a_n\to a_0$.
We verify the two conditions in the
definition of Mosco convergence.
\vskip0.1cm
\par We first verify \ref{M1}. Assume \(0\leq a_0<1\) and let $v_n\in H_0^1(D_{a_n})$ satisfy
\begin{equation}\label{weakly1}
    v_n\rightharpoonup v \quad \text{weakly in } H_0^1(B).
\end{equation}
We need to prove that $v\in H_0^1(D_{a_0})$. Choose a strictly decreasing sequence \(\{b_m\}\subset(a_0,1)\)
such that $b_m\downarrow a_0$. For each fixed \(m\), since \(a_n\to a_0<b_m\), there exists
\(N_m\in\mathbb N\) such that $a_n<b_m$ for all $n\geq N_m$. Thus,
$$S_{b_m}\subset S_{a_n},~ D_{a_n} \subset D_{b_m}, \text{ and } H_0^1(D_{a_n})
\subset
H_0^1(D_{b_m}).$$
The space $H_0^1(D_{b_m})$ is a closed linear subspace of
$H_0^1(B)$ and is therefore weakly closed. Hence \eqref{weakly1} gives $v\in H_0^1(D_{b_m})$
for any fixed $m$.
Let $\widetilde v$ denote the quasi-continuous representative of
$v$. 
The capacitary characterization of \(H_0^1(D_{b_m})\) yields
\[
\widetilde v=0
\quad \text{quasi-everywhere on }S_{b_m}~~\text{for every }m.\] For each $m$, there exists
$N_m'\subset S_{b_m}$ with
$\operatorname{Cap}_2(N_m',B)=0$
such that $\widetilde v=0$ on $S_{b_m}\setminus N_m'$. By countable
subadditivity of the relative Sobolev capacity,
\[
\operatorname{Cap}_2\left(\bigcup_{m=1}^{\infty}N_m',B\right)=0.
\]
Moreover,  $\displaystyle \bigcup_{m=1}^{\infty}S_{b_m}
=
S_{a_0}\setminus
\bigl\{(-a_0,0),(a_0,0)\bigr\}$.
Since points have zero $H^1$-capacity in dimension two
(see \cite{Mazya2011}),
\[
\operatorname{Cap}_2
\bigl(\{(-a_0,0),(a_0,0)\},B\bigr)=0.
\]
Consequently,
\[
\widetilde v=0
\quad\text{quasi-everywhere on }S_{a_0}.\]
Hence the capacitary characterization of $H_0^1$ gives
$v\in H_0^1(D_{a_0})$.
If $a_0=1$, then $D_{a_0}=B$ and the weak limit already satisfies
$v\in H_0^1(B)=H_0^1(D_1)$. Thus the Mosco-liminf condition holds for
every $a_0\in[0,1]$.
\vskip0.1cm

 Next, we verify \ref{M2}. Assume  that $0\leq a_0<1$. By a standard diagonal argument and the density of
$C_c^\infty(D_{a_0})$ in $H_0^1(D_{a_0})$, it suffices to verify
the recovery condition for $v\in C_c^\infty(D_{a_0})$. Since $\operatorname{supp}v\subset \subset D_{a_0}$ and $a_n\to a_0$, we have  $\operatorname{supp}v\cap S_{a_n}=\varnothing$  for all sufficiently large $n$. Hence 
\[ v\in C_c^\infty(D_{a_n})\subset H_0^1(D_{a_n}). \]
The recovery condition follows by taking $v_n=v$. The case $a_0=1$ is similar. For every $v\in C_c^\infty(B)$, the compactness of $\operatorname{supp}v$ and $a_n\to1$ imply that $\operatorname{supp}v\cap S_{a_n}=\varnothing$ for all sufficiently large $n$. This proves \ref{M2}.
Both Mosco conditions are therefore satisfied, and
\eqref{Mosco1-1} follows.
\end{proof}

\begin{Lem}\label{Mosco1}
For any $a \in (0,1)$, as $\epsilon \to 0$, we have
\[
H_0^1(\Omega_{a,\epsilon})
\xrightarrow{M}
H_0^1(D_a).
\]
\end{Lem}
\begin{proof}
We verify the two Mosco conditions. For \ref{M1}, let $\epsilon_n \to 0$ and  $v_n \in H_0^1(\Omega_{a,\epsilon_n})$
satisfy
\[
v_n \rightharpoonup v
\quad \text{weakly in } H_0^1(B).
\]
Since $H_0^1(D_a)$ is a closed linear subspace of $H_0^1(B)$,
it is weakly closed. Hence
$v\in H_0^1(D_a)$
and \ref{M1} follows.
For \ref{M2}, by a standard diagonal
argument and the density of $C_c^\infty(D_a)$ in $H_0^1(D_a)$,
it suffices to verify the recovery condition for
$\varphi\in C_c^\infty(D_a)$. Note that \ref{C4} implies
\[
\Omega_{a,\epsilon}\uparrow D_a
\quad \text{as } \epsilon\to 0.
\]
Since $\operatorname{supp}\varphi\subset\subset D_a$, the compactness of
$\operatorname{supp}\varphi$ yields $\epsilon_\varphi>0$ such that
\[
\operatorname{supp}\varphi\subset\Omega_{a,\epsilon}
\quad
\text{whenever }0<\epsilon<\epsilon_\varphi.
\]
Therefore, for all sufficiently small $\epsilon$,
\[
\varphi\in C_c^\infty(\Omega_{a,\epsilon})
\subset H_0^1(\Omega_{a,\epsilon}).
\]
Thus the recovery condition
holds on the dense subspace $C_c^\infty(D_a)$, and hence \ref{M2}
follows by the standard diagonal argument.
Thus,
\[
H_0^1(\Omega_{a,\epsilon})
\xrightarrow{M}
H_0^1(D_a)
\quad \text{as } \epsilon\to 0.
\]
\end{proof}

\subsection{Results on local regularity}$\,$\vskip 0.1cm
We conclude the section with two local regularity consequences of the preceding Mosco convergence results.

 \begin{Lem}\label{smooth-convergence}
For every $a\in(0,1)$, we have
\[
        u_{a,\eps}\to u_a\quad\hbox{strongly in }H^1_0(B)
        \quad\text{as }\eps\to0.
\]
Moreover, for every compact set $K\subset \subset D_a$, we have
\[
        u_{a,\eps}\to u_a\quad\text{in }C^2(K)\quad\text{as }\eps\to0.
\]
\end{Lem}
\begin{proof}
By  \ref{C4}, we have $\Omega_{a,\epsilon}\uparrow D_a$ as $\epsilon\to 0$. 
We view $H^1_0(\Omega_{a,\epsilon})$ and $H^1_0(D_a)$ as closed subspaces of $H^1_0(B)$ via zero extension.
The functions $u_{a,\epsilon}$ and $u_a$ are the unique minimizers of $J$ over $H^1_0(\Omega_{a,\epsilon})$ and $H^1_0(D_a)$, respectively, where $J$ is given in \eqref{Ju}.
The Euler--Lagrange equation for $u_{a,\epsilon}$ is
\begin{equation}\label{Euler}
    \int_B \nabla u_{a,\epsilon} \cdot \nabla v \, dz
    = \int_B v \, dz,
    \quad \text{for all } v \in H^1_0(\Omega_{a,\epsilon}).
\end{equation}
Choosing $v=u_{a,\epsilon}$ in \eqref{Euler} yields
\[\displaystyle\int_B |\nabla u_{a,\epsilon}|^2 \, dz = \int_B u_{a,\epsilon} \, dz .\]
By  H\"older's inequality and  Poincar\'e's inequality, 
\[
\int_B |\nabla u_{a,\epsilon}|^2 \, dz
\le |B|^{1/2} \|u_{a,\epsilon}\|_{L^2(B)}
\le C \|\nabla u_{a,\epsilon}\|_{L^2(B)} .
\]
Hence $\|u_{a,\epsilon}\|_{H^1_0(B)}\le C$ uniformly for all sufficiently small $\epsilon>0$.
Let $\epsilon_n\to 0$ be arbitrary. Passing to a subsequence,
still denoted by $\epsilon_n$, we may assume that
\[
u_{a,\epsilon_n}\rightharpoonup u
\quad\text{weakly in }H_0^1(B)
\]
for some $u\in H_0^1(B)$. Lemma~\ref{Mosco1} and \ref{M1} yield $u\in H_0^1(D_a)$.
Since $J$ is weakly lower semicontinuous on $H_0^1(B)$,
\[
J(u)
\le
\liminf_{n\to\infty}J(u_{a,\epsilon_n}).
\]
On the other hand, by \ref{M2}, there exists
$w_n\in H_0^1(\Omega_{a,\epsilon_n})$ such that
\[
w_n\to u_a
\quad\text{strongly in }H_0^1(B).
\]
Since $u_{a,\epsilon_n}$ minimizes $J$ over
$H_0^1(\Omega_{a,\epsilon_n})$, we have $J(u_{a,\epsilon_n})\le J(w_n)$.
Therefore,
\[
J(u)
\le
\liminf_{n\to\infty}J(u_{a,\epsilon_n})
\le
\limsup_{n\to\infty}J(u_{a,\epsilon_n})
\le
J(u_a).
\]
Since $u\in H_0^1(D_a)$ and $u_a$ is the unique minimizer of $J$
over $H_0^1(D_a)$, we also have
$J(u_a)\le J(u)$.
Consequently,
$u=u_a$
and
$J(u_{a,\epsilon_n})\to J(u_a)$.
The compact embedding
$H_0^1(B)\hookrightarrow L^2(B)$ gives
\[
u_{a,\epsilon_n}\to u_a
\quad\text{strongly in }L^2(B).
\]
In particular,
$\displaystyle\int_B u_{a,\epsilon_n}\,dx
\to
\int_B u_a\,dx$.
Therefore
\[
\|\nabla u_{a,\epsilon_n}\|_{L^2(B)}^2
\to
\|\nabla u_a\|_{L^2(B)}^2.
\]
Together with the weak convergence, this yields
\[
u_{a,\epsilon_n}\to u_a
\quad\text{strongly in }H_0^1(B).
\]
Since the sequence $\{\epsilon_n\}$ was arbitrary, we conclude that
\[
u_{a,\epsilon}\to u_a
\quad\text{strongly in }H_0^1(B).
\]
It remains to prove the local $C^2$ convergence. Let $K\subset\subset D_a$
and choose an open set $U$ such that
$K\subset \subset U\subset \subset D_a$.
By \ref{C4} and the compactness of $\overline{U}$, we have
$\overline{U}\subset\Omega_{a,\epsilon}$ for all sufficiently small
$\epsilon>0$. Hence
\[
-\Delta(u_{a,\epsilon}-u_a)=0
\quad\text{in }U.
\]
By the interior estimates for harmonic functions,
\[
\|u_{a,\epsilon}-u_a\|_{C^{2,\alpha}(K)}
\le
C_{K,U,\alpha}\,
\|u_{a,\epsilon}-u_a\|_{L^2(U)}.
\]
The right-hand side tends to zero as $\epsilon\to0$.
Therefore,
\[
u_{a,\epsilon}\to u_a
\quad\text{in }C^2(K).
\]
\end{proof}
\begin{Lem}\label{continuity-a}
The map $a\mapsto u_a$ is continuous from $[0,1]$ into $H_0^1(B)$.
Moreover, for every $a_0\in[0,1]$ and every
$K\subset\subset D_{a_0}$,
\[
u_a\to u_{a_0}
\quad\text{in }C^2(K)
\quad\text{as }a\to a_0.
\]
Consequently, $\mu$ and $\Lambda$, defined in \eqref{mu(a)}
and \eqref{Lambda-formula}, respectively, are continuous on $(0,1]$.
\end{Lem}

\begin{proof}
Fix $a_0\in[0,1]$ and let $a_n\to a_0$. By
Lemma~\ref{Mosco2},
\[
H_0^1(D_{a_n})\xrightarrow{M}H_0^1(D_{a_0}).
\]
Repeating the variational argument in the proof of
Lemma~\ref{smooth-convergence}, with $D_{a_n}$ in place of
$\Omega_{a,\epsilon_n}$, gives
\[
u_{a_n}\to u_{a_0}
\quad\text{strongly in }H_0^1(B).
\]
Since the sequence $a_n\to a_0$ was arbitrary, this proves the first assertion.

Now let $K\subset\subset U\subset\subset D_{a_0}$.
When the parameter $a$ is sufficiently close to $a_0$, we have $\overline U\subset D_a$, and hence
\[
-\Delta(u_a-u_{a_0})=0\quad\text{in }U.
\]
Thus, for every $\alpha\in(0,1)$,
\[
\|u_a-u_{a_0}\|_{C^{2,\alpha}(K)}
\le C_{K,U,\alpha}
   \|u_a-u_{a_0}\|_{L^2(U)}
\to 0.
\]
This proves the local $C^2$ convergence. For $a_0\in(0,1]$, choosing $K\subset\subset D_{a_0}$ with
$0\in \operatorname{int}K$ yields
\[
\partial_{yy}u_a(0)\to\partial_{yy}u_{a_0}(0) \quad \text{as }a\to a_0.
\]
Then $\mu(a)\to\mu(a_0)$. Thus $\mu$ is continuous on $(0,1]$, and the continuity of
$\Lambda$ follows from \eqref{Lambda-formula}.
\end{proof}

\section{Change of Hessian signature at the origin}\label{sign change}
In this section, we consider the Hessian at the origin of the torsion function $u_a$ on $\Omega = D_a$. Proposition \ref{small-a-positive} and Proposition \ref{large-a-negative} show that $\Lambda(a)$ is positive for small $a$ and negative for $a$ close to $1$. Furthermore, Proposition~\ref{small-a-positive} will also be used to prove Theorem~\ref{Thm1.2} and Corollary~\ref{cor1.3}.

\begin{Prop}\label{small-a-positive}
As $a\downarrow0$, 
\begin{equation}\label{mu-positive-small-a}
        \mu(a)=\partial_{yy}u_a(0)=\frac{\beta}{a}+o\left(\frac1a\right)
\end{equation}
for $\beta:=\partial_yW(0,0)$, where $W$ is the solution of
\begin{equation}\label{WWWW}
    -\Delta W=1\quad\text{in }B^+,
        \quad W=0\quad\text{on }\partial B^+.
\end{equation}
In particular, $\Lambda(a)>0$ for  sufficiently small $a$.
\end{Prop}

\begin{Lem}
Let $z = x + iy$ and define $P_a(x,y)$ in the upper half-plane by
\begin{equation}\label{defP_a}
    P_a(x,y) := \mathrm{Re}\,\sqrt{a^2 - z^2}-y,
\end{equation}
where the analytic branch of
$\displaystyle \sqrt{a^2-z^2}$ is chosen such that
$\sqrt{a^2-x^2}>0$
for all $x\in(-a,a)$.
Then $P_a$ satisfies 
\begin{equation*}
    \begin{cases}
\quad \; \; -\Delta P_a = 0 & \text{in } {B}^+,\\
\quad  P_a(x,0) = 0, & |x|>a,\\
\partial_\nu P_a(x,0) = 1, & |x|<a.
\end{cases}
\end{equation*}
Moreover, there exists $C>0$ such that, for any $a\in(0,\frac{1}{2})$, we have
\begin{equation}\label{P_a}
    |P_a| \le C a^2 \quad \text{on }  \partial {B}^+ \cap \{y>0\}.
\end{equation}
\end{Lem}
\begin{proof}
 Since the map $\displaystyle z \mapsto \sqrt{a^2-z^2}$
is analytic in the upper half-plane, its real part
is harmonic, and hence $\displaystyle -\Delta P_a=0$ in $ \displaystyle {B}^+$.
When \(y=0\) and \(|x|>a\),  we have $a^2-x^2<0$
and $\displaystyle \operatorname{Re}\sqrt{a^2-z^2}=0$,
thus $P_a(x,0)=0$ for all $|x|>a$.
For \(y=0\) and \(|x|<a\), we have $\displaystyle \operatorname{Re}\sqrt{a^2-z^2}
=
\sqrt{a^2-x^2}$,
hence  $\displaystyle \partial_\nu P_a(x,0)=1$.
It remains to prove the estimate on the semicircular boundary. 
For \(0<a< \frac12\) and \(|z|=1\), the choice of branch gives $\displaystyle \sqrt{a^2-z^2}
=
-iz\sqrt{1-\frac{a^2}{z^2}}$. Then a  Taylor expansion yields 
\begin{equation}\label{genhaotailezhankai}
\sqrt{a^2-z^2}
=
-iz+\frac{i a^2}{2z}+O(a^4), \quad \text{uniformly for } z\in\partial B^+\cap\{y>0\}.
\end{equation}
Taking real parts in \eqref{genhaotailezhankai}, we get
$\displaystyle \operatorname{Re}\sqrt{a^2-z^2}
=
y+O(a^2)$.
Thus
\[
|P_a|\leq Ca^2
\quad\text{on }\partial B^+\cap\{y>0\},
\]
where \(C>0\) is independent of \(a\in(0,\frac{1}{2})\).
\end{proof}

\begin{Lem}
There exists  a unique weak solution \(\eta_a\in H^1(B^+)\) of
\begin{equation}\label{eta_a}
    \begin{cases}
-\Delta \eta_a = 0 & \text{in } {B}^+,\\
\quad \; \; \eta_a = -P_a & \text{on } \Gamma_a,\\
\; \; \partial_\nu \eta_a = 0 & \text{on } I_a,
\end{cases}
\end{equation}
where $I_a := \bigl(-a,a\bigr)\times\{0\}$ and $\Gamma_a = \partial {B}^+ \setminus I_a$. Moreover,  there exists  \(C>0\) independent of \(a\) such that
\begin{equation}\label{etasmall}
    \|\eta_a\|_{L^\infty({B}^+)} \leq C a^2.
\end{equation}
\end{Lem}

\begin{proof}
The existence and uniqueness follow from the standard variational theory for mixed Dirichlet--Neumann
problems (see
\cite{GT1983,
grisvard2011elliptic}).
By  \eqref{P_a}, 
\[
|\eta_a| = |-P_a| \leq C a^2 \quad \text{on } \partial {B}^+ \cap \{y>0\}.
\]
By the boundary condition in the mixed problem and the estimate for $P_a$, \[|\eta_a| \leq C a^2 \quad \text{on } \Gamma_a.\]
Define $$w := (\eta_a - C a^2)^+ ~~\text{ and }~~ V_a := \{v\in H^1({B}^+) : v=0 \text{ on } \Gamma_a\}.$$
Then we have $w \in V_a$.
By \eqref{eta_a}, for any \(\varphi \in V_a\), we have
\begin{equation}\label{fangchengruojie}
    0
=
\int_{{B}^+} (\Delta \eta_a)\varphi\,dx
=
-\int_{{B}^+} \nabla \eta_a \cdot \nabla \varphi\,dx
+
\int_{\partial {B}^+} \partial_\nu \eta_a\,\varphi\,ds=-\int_{{B}^+} \nabla \eta_a \cdot \nabla \varphi\,dx .
\end{equation}
Taking \(\varphi=w\) in \eqref{fangchengruojie} yields
$\displaystyle \int_{{B}^+} \nabla \eta_a \cdot \nabla w\,dx = 0$,
hence $\displaystyle \int_{{B}^+} |\nabla w|^2\,dx = 0$.
 Since $w=0$ on  $\Gamma_a$, 
\begin{equation}\label{yifang}
    \eta_a \leq C a^2
\quad \text{a.e. in } {B}^+.
\end{equation}
 Let $\displaystyle \overline w := (-Ca^2-\eta_a)^+$.
Similarly, we obtain
\begin{equation}\label{linyifang}
    \eta_a \geq -C a^2
\quad \text{a.e. in } {B}^+.
\end{equation}
Therefore, \eqref{yifang} and \eqref{linyifang} give \eqref{etasmall}.
\end{proof}

\begin{proof}[\bf{Proof of Proposition  \ref{small-a-positive}}]
The Hopf lemma gives $\beta > 0$.
 By symmetry, $v_a:=u_a|_{B^+}$ solves the mixed problem
\begin{equation*}
\begin{cases}
        -\Delta v_a=1 & \text{in} ~B^+,\\
        \quad \; \; v_a=0&  \text{on}~\partial B^+\setminus I_a,\\
       \; \; \partial_\nu v_a=0& \text{on}~ I_a,
\end{cases}
\end{equation*}
where  $I_a := \bigl(-a,a\bigr)\times\{0\}$ and $\nu=(0,-1)$ on the flat part of $\partial B^+$.
  Set \begin{equation}\label{hadef}
    h_a:=v_a-W,
\end{equation} 
where $W$ is the solution of \eqref{WWWW}.
Then $h_a$ satisfies
\begin{equation*}
\begin{cases}
        -\Delta h_a=0&  \text{in}~B^+,\\
        \quad \; \; h_a=0 & \text{on}~ \partial B^+\setminus I_a,\\
       \; \; \partial_\nu h_a=\partial_{y}W(x,0)  &\text{on } ~ I_a.
\end{cases}
\end{equation*}
Set $$g(x):=\partial_{y}W(x,0)-\beta.$$
Then  $\partial_{\nu}h_{a}
=
\beta+g(x) ~~\text{on } I_{a}$.  By the symmetry of $B^+$ with respect to the $y$-axis and the uniqueness
of the Dirichlet problem,
$W(-x,y)=W(x,y)$.
Hence $g$ is even. In particular,
$g(0)=g'(0)=0$. By the smooth boundary regularity of $W$ near the origin, there exist
$r_0>0$ and $C>0$ such that
\begin{equation}\label{g-quadratic}
|g(x)|\le Cx^2,
\quad
|g'(x)|\le C|x|
\quad\text{for all } |x|<r_0.
\end{equation} Consequently, for all sufficiently small $a>0$,
\begin{equation}\label{gxiaoyudengyuCa}
    |g(x)|\leq C a^2
\quad\text{for }x\in(-a,a).
\end{equation}
Let $Q_a=P_a+\eta_a$, where $P_a$ is defined by \eqref{defP_a} and $\eta_a$ is given by \eqref{eta_a}. Then \(Q_a\) is the
unique weak solution of
\begin{equation*}
\begin{cases}
        -\Delta Q_a=0 & \text{in}~B^+,\\
        \quad \; \;Q_a=0& \text{on}~\partial B^+\setminus I_a,\\
        \; \; \partial_\nu Q_a=1& \text{on}~I_a.
\end{cases}
\end{equation*}
 For any test function \(\varphi\) satisfying $\varphi\big|_{\Gamma_a}=0$, we have \[
\int_{{B}^+} \nabla Q_a\cdot \nabla \varphi \, dx
=
\int_{I_a} \varphi \, ds .
\]
Taking $\varphi = Q_a^- := \max\{-Q_a,0\}$, we get
\[
-\int_{{B}^+} |\nabla Q_a^-|^2 \, dx
=
\int_{I_a} Q_a^- \, ds
\ge 0 .
\]
Therefore, $Q_a\ge 0$ in ${B}^+$.
By the strong maximum principle, we obtain
\begin{equation}\label{Qadayu0}
    Q_a>0 \quad \text{in } {B}^+.
\end{equation}
Set \begin{equation}\label{Eadef}
    E_a:=h_a-\beta Q_a.
\end{equation} Then $E_a$ is the unique weak solution of \begin{equation*}
    \begin{cases}
-\Delta E_{a}=0 & \text{in }B^{+},\\
\quad \; \; E_{a}=0 & \text{on }\partial B^{+}\setminus I_{a},\\
\; \; \partial_{\nu}E_{a}=g(x) & \text{on }I_{a}.
\end{cases}
\end{equation*}
Choose $C>0$ so that \eqref{gxiaoyudengyuCa} holds. Testing the weak
equations with
\[(E_a-Ca^2Q_a)^+
\quad \text{and} \quad
(-E_a-Ca^2Q_a)^+,\]
 respectively, yields
\[E_a\le Ca^2Q_a
\quad \text{and} \quad
-E_a\le Ca^2Q_a.\]
Together with \eqref{Qadayu0}, this yields
\begin{equation}\label{Ea-small}
        |E_a|\le C a^2 Q_a\quad\text{in }B^+.
\end{equation}
We next rescale around the origin. Fix
$0<\rho<\rho'<1$.
Define
\[
\widehat\eta_a(X,Y)
:=
\frac{\eta_a(aX,aY)}{a}.
\]
Since $\partial_\nu\eta_a=0$ on $I_a$, the even reflection
$\widehat\eta_a^*(X,Y)
:=
\widehat\eta_a(X,|Y|)$
is weakly harmonic in $B_{\rho'}$ for all sufficiently small $a>0$.
By \eqref{etasmall},
\[
\|\widehat\eta_a^*\|_{L^\infty(B_{\rho'})}
\le Ca.
\]
The interior estimates for harmonic functions therefore imply \begin{equation}\label{Pazhengze}
    \|\widehat \eta_{a}\|_
{C^{2,\alpha}(\overline{B_{\rho}^{+}})}
\leq
Ca,
\end{equation}
where $C>0$ depends only on $\rho$, $\rho'$ and $\alpha$.
Set
 $$\Phi(X,Y)
:=
\operatorname{Re}\sqrt{1-(X+iY)^{2}}.$$
From  $Q_a=P_a+\eta_a$, the definition of $P_a$  and the estimate \eqref{Pazhengze}, we obtain \begin{equation}\label{Qalim}
    \frac{Q_{a}(aX,aY)}{a}
\to
\Phi(X,Y)-Y \quad  \text{in } C^{2,\alpha}\Bigl(\overline{B_{\rho}^{+}}\Bigl) \quad\text{as }a \to 0.
\end{equation}
Moreover, using the same identity and
$\|\eta_a\|_{L^\infty(B^+)}\le Ca^2$, we obtain
\begin{equation}\label{Qa-rescaled-bound}
\left\|
\frac{Q_a(a\,\cdot)}{a}
\right\|_{L^\infty(B_{\rho'}^+)}
\le C_{\rho'}.
\end{equation}
Note that \eqref{Qa-rescaled-bound} is obtained directly on
$B_{\rho'}^+$ and does not require enlarging the radius in
\eqref{Qalim}.
Next define
\[
\widehat E_a(X,Y)
:=
\frac{E_a(aX,aY)}{a},
\quad
G_a(X):=g(aX).
\]
Then
\begin{equation}\label{Eamixed}
    \begin{cases}
-\Delta\widehat E_a=0
& \text{in }B_{\rho'}^+,\\
\; \; \partial_\nu\widehat E_a=G_a
& \text{on }\{Y=0,\ |X|<\rho'\}.
\end{cases}
\end{equation}
Since $g(0)=g'(0)=0$ and $g$ is smooth near the origin,
\eqref{g-quadratic} gives
\begin{equation*}
    \| G_{a}\|_
{C^{1,\alpha}(-\rho',\rho')}
\leq Ca,
\end{equation*} where $C>0$ depends only on $\rho'$ and $\alpha$.
 Moreover,  \eqref{Ea-small} and \eqref{Qa-rescaled-bound} imply \begin{equation*}
    \|\widehat E_{a}\|_{L^{\infty}(B_{\rho'}^{+})}
\leq Ca.
\end{equation*} 
By the local Schauder estimate for the Neumann problem for \eqref{Eamixed}
(see \cite{GT1983}), \begin{equation}\label{EaC2a}
    \|\widehat E_{a}\|_
{C^{2,\alpha}(\overline{B_{\rho}^{+}})}
\leq
C
\big(
\|\widehat E_{a}\|_{L^{\infty}(B_{\rho'}^{+})}
+
\| G_{a}\|_
{C^{1,\alpha}(-\rho',\rho')}
\big)
\leq
Ca,
\end{equation}
where $C>0$ depends only on $\rho$, $\rho'$ and $\alpha$. Finally, define $$\displaystyle \widehat W_{a}(X,Y)
:=
\frac{W(aX,aY)}{a}.$$ Since $W(0,0)=\partial_{x}W(0,0)=0$, $\partial_{y}W(0,0)=\beta$, a Taylor expansion gives
\begin{equation}\label{Walim}
    \widehat W_{a}
\to
\beta Y
\quad\text{in }
C^{2}\bigl(\overline{B_{\rho}^{+}}\bigr).
\end{equation}
By \eqref{hadef} and \eqref{Eadef}, we have $$v_{a}=W+\beta Q_{a}+E_{a}.$$
Combining \eqref{Qalim}, \eqref{EaC2a} and \eqref{Walim},  we obtain
\begin{equation*}
    \frac{v_{a}(aX,aY)}{a}
\to
\beta\Phi(X,Y) \quad \text{in } C^{2}\bigl(\overline{B_{\rho}^{+}}\bigr).
\end{equation*} Both $u_a(aX,aY)/a$ and $\Phi(X,Y)$ admit even $C^2$ extensions
across $\{Y=0,\ |X|<\rho\}$. Therefore the preceding convergence
extends to a full neighborhood of the origin. Since
$\partial_{YY}\Phi(0,0)=1$
and
\[
\partial_{YY}
\left(\frac{u_a(aX,aY)}{a}\right)\bigg|_{(0,0)}
=
a\,\partial_{yy}u_a(0,0),
\]
we conclude that
$a\,\partial_{yy}u_a(0,0)\to\beta ~~ \text{as }a \to 0$,
hence \eqref{mu-positive-small-a} follows.
\end{proof}

Let $\displaystyle U:=\frac{1-x^2-y^2}{4}$. Then $U$ is the torsion function of $B$, i.e.,  \begin{equation} \label{UU}
 -\Delta U=1\; \text{ in }B,
        \quad U=0 \; \text{ on }\partial B.
\end{equation}
\begin{Prop}\label{large-a-negative}
As $a\uparrow1$, one has
\[
        u_a\to U
        \quad\text{in }C^2_{\rm loc}(B).
\]
Consequently, 
\[
        D^2u_a(0)\to -\frac12 I_2,
        \qquad \Lambda(a)\to -\frac12,
\]
and hence $\Lambda(a)<0$ for all $a$ sufficiently close to $1$.
\end{Prop}
\begin{proof}
This is precisely the case $a_0=1$ of Lemma \ref{continuity-a}.
\end{proof}
Set
\begin{equation*}
        E:=\{a\in(0,1):\Lambda(a)\ge0\}
        , ~~~~ a^*:=\sup E
\end{equation*} and likewise
\begin{equation*}
        F:=\{a\in(0,1):\mu(a)\le0\},~~~~
         b^*:=\inf F.
\end{equation*}
 Proposition \ref{small-a-positive} and Proposition \ref{large-a-negative} show that $E,F\neq\varnothing$ and $a^*$, $b^*$ lie in $(0,1)$. Since $\Lambda(a)$ and $\mu(a)$ are continuous on $(0,1)$, we have
\begin{equation}\label{critical-lambda-zero}
        \Lambda(a^*)=0 \text{ and }
        \Lambda(a)<0 \quad\hbox{for every }a\in(a^*,1).
\end{equation}
Moreover,
\begin{equation}\label{critical-lambda-zeroF}
        \mu(a)>0  \quad\hbox{for every }a\in(0,b^*).
        \end{equation}
\begin{Prop}
\label{prop:saddle-persistence}
Fix $a\in(0,b^*)$. Then there exists $\delta(a)>0$ such that for all $\epsilon\in(0,\delta(a))$, the origin is a nondegenerate saddle point of $u_{a,\epsilon}$.
More precisely, letting $\mu_{a,\epsilon}:=\partial_{yy}u_{a,\epsilon}(0,0)$, we have $\mu_{a,\epsilon}>0$ and
\begin{equation}
\label{eq:Hessian-smooth-domain}
D^{2}u_{a,\epsilon}(0,0)=
\begin{pmatrix}
-1-\mu_{a,\epsilon} & 0\\
0 & \mu_{a,\epsilon}
\end{pmatrix}.
\end{equation}
\end{Prop}
\begin{proof}
For any fixed $a\in(0,b^*)$, we can choose $r_a>0$ such that $\displaystyle \overline{B_{2r_a}(0)}\subset \subset D_a$. Since $a\in(0,b^*)$,  \eqref{critical-lambda-zeroF} gives $\mu(a)>0$.
By Lemma \ref{smooth-convergence}, it follows that
\begin{equation*}
   \mu_{a,\epsilon}
:= \partial_{yy}u_{a,\epsilon}(0,0)
\to
\partial_{yy}u_a(0,0)
=
\mu(a)>0.
\end{equation*}
Hence, there exists $\delta(a)>0$ such that
\begin{equation}\label{mu_{a,epsilon}}
\mu_{a,\epsilon}
\geq
\frac{\mu(a)}{2}>0 \quad \text{for all } \epsilon \in (0,\delta(a)).
\end{equation}
The symmetry of $u_{a,\epsilon}$ with respect to the coordinate
axes implies
\begin{equation}\label{symmetry-u_e}
\partial_xu_{a,\epsilon}(0,0)
=
\partial_yu_{a,\epsilon}(0,0)
=
\partial_{xy}u_{a,\epsilon}(0,0)
=
0.
\end{equation}
Since $-\Delta u_{a,\epsilon}=1$, we have
\[
\partial_{xx}u_{a,\epsilon}(0,0)
=
-1-\mu_{a,\epsilon}<0,
\]
which proves \eqref{eq:Hessian-smooth-domain}. Therefore the Hessian  has one
strictly negative and one strictly positive eigenvalue, so the origin
is a nondegenerate saddle point.
\end{proof}

\begin{proof}[\bf{Proof of Theorem \ref{Thm1.2}}]
Fix $a\in(0,b^*)$,  let $\delta(a)$ be as in
Proposition~\ref{prop:saddle-persistence}, and take $\displaystyle \epsilon\in \big(0,\delta(a)\big)$. By \eqref{symmetry-u_e} and the Taylor expansion, 
\[
u_{a,\epsilon}(0,y)
-
u_{a,\epsilon}(0,0)
=
\frac{\mu_{a,\epsilon}}{2}y^{2}
+
o(y^{2}),
~~\text{as }y\to0.
\]
Since $a \in (0,b^*)$, \eqref{mu_{a,epsilon}} yields $\mu_{a,\epsilon}>0$. Therefore, there exists $r_{a,\epsilon}>0$ such that
\begin{equation}\label{eq:vertical-growth}
u_{a,\epsilon}(0,y)>u_{a,\epsilon}(0,0)
\quad
\text{for every }0<|y|<r_{a,\epsilon}.
\end{equation}
In particular,
$\max\limits_{\overline{\Omega_{a,\epsilon}}}u_{a,\epsilon}
>
u_{a,\epsilon}(0,0).$
Since $u_{a,\epsilon}\in C(\overline{\Omega_{a,\epsilon}})$,
$u_{a,\epsilon}$  attains its global maximum. Moreover,
$u_{a,\epsilon}>0$ in $\Omega_{a,\epsilon}$ and
$u_{a,\epsilon}=0$ on $\partial\Omega_{a,\epsilon}$, so every
global maximizer lies in $\Omega_{a,\epsilon}$.
By Lemma~\ref{lem:horizontal-monotonicity}, every global maximizer
lies on the $y$-axis. Hence there exists $y_\epsilon$ such that
\[
u_{a,\epsilon}(0,y_\epsilon)
=
\max_{\overline{\Omega_{a,\epsilon}}}u_{a,\epsilon}.
\]
By \eqref{eq:vertical-growth}, we have $y_\epsilon\neq0$.
Since $\Omega_{a,\epsilon}$ is symmetric with respect to the
$x$-axis, uniqueness of the Dirichlet problem implies
$u_{a,\epsilon}(x,-y)=u_{a,\epsilon}(x,y)$.
Consequently,
\[
u_{a,\epsilon}(0,-y_\epsilon)
=
u_{a,\epsilon}(0,y_\epsilon)
=
\max_{\overline{\Omega_{a,\epsilon}}}u_{a,\epsilon}.
\]
Since $y_\epsilon\neq0$, the points
$(0,y_\epsilon)$ and $(0,-y_\epsilon)$ are distinct global
maximizers of $u_{a,\epsilon}$.
\end{proof}

\begin{proof}[\bf{Proof of Corollary \ref{cor1.3}}]
Fix $a\in(0,b^*)$. Let $\delta(a)>0$ be given by
Proposition~\ref{prop:saddle-persistence}. For every
$0<\epsilon<\delta(a)$, \eqref{eq:vertical-growth}
implies that there exists $t_\epsilon>0$ sufficiently small such that
    $u_{a,\epsilon}(0,\pm t_\epsilon)
    >
    u_{a,\epsilon}(0,0)$.
Thus
\[
\mathcal{L}_{a,\epsilon}^{+}
:=
\mathcal{L}_{a,\epsilon}\cap\{y>0\}
\text{ and }
\mathcal{L}_{a,\epsilon}^{-}
:=
\mathcal{L}_{a,\epsilon}\cap\{y<0\}
\]
are both nonempty.
By Lemma~\ref{lem:horizontal-monotonicity} and the definition of
$\mathcal L_{a,\epsilon}$ in \eqref{1.7},
\[
    \mathcal L_{a,\epsilon}\cap\{y=0\}
    =
    \varnothing.
\]
Consequently,
\[
    \mathcal L_{a,\epsilon}
    =
    \mathcal L_{a,\epsilon}^{+}
    \cup
    \mathcal L_{a,\epsilon}^{-},
\]
where $\mathcal L_{a,\epsilon}^{+}$ and
$\mathcal L_{a,\epsilon}^{-}$ are nonempty, disjoint,
and relatively open in $\mathcal L_{a,\epsilon}$. Hence $\mathcal L_{a,\epsilon}$ is
disconnected. In particular, $\mathcal L_{a,\epsilon}$ cannot be star-shaped.
\end{proof}

\section{Global maximizer and proof of Theorem~\ref{Thm1.1}} \label{global maximum}

In this section, we prove Lemma~\ref{green-monotonicity}, which establishes the strict monotonicity of $u_a$ along the $y$-axis if $\mu(a)<0$. Using this monotonicity, Proposition~\ref{origin-global-max-slit} shows that the origin is the unique global maximizer of $u_a$ in $D_a$ for $a\in(a^*,1)$. Finally, we complete the proof of Theorem~\ref{Thm1.1}.

\begin{Lem}\label{green-monotonicity}
Let $a\in(0,1)$.  $\mu(a)=\partial_{yy}u_a(0)$ is defined by \eqref{mu(a)}.  If $\mu(a)<0$, then
\begin{equation*}
        \partial_yu_a(0,y)<0 \quad\text{for }0<y<1.
\end{equation*}
Hence $\displaystyle y\mapsto u_a(0,y)$ is strictly decreasing on $[0,1)$, and strictly increasing on $(-1,0]$ by symmetry.
\end{Lem}

\begin{proof}
Let $H_a:=U-u_a$, where $U$ is defined in \eqref{UU}.
Then $H_a \in H_0^1(B)$ and is harmonic in $D_a$. The maximum principle gives $ H_a \ge 0$  in $D_a$. Define
$$H_{a,\epsilon}:=U-u_{a,\epsilon}\text{ and }
\Sigma_{a,\epsilon}
:=
\partial T_{a,\epsilon}\cap B,$$
 where \(u_{a,\epsilon}\) is the
torsion function of \(\Omega_{a,\epsilon}\), extended by zero to
\(B\). 
For every \(\varphi\in C_c^\infty(B)\), integration by parts gives
\begin{equation*}
\left\langle-\Delta H_{a,\epsilon},\varphi\right\rangle
=
\int_{T_{a,\epsilon}}\varphi\,dx
-
\int_{\Sigma_{a,\epsilon}}
\partial_\nu u_{a,\epsilon}\,
\varphi\,d\mathcal H^1,
\end{equation*}
where \(\nu\) is the outward unit normal of
\(\Omega_{a,\epsilon}\), and
\(\mathcal H^1\) is the one-dimensional Hausdorff measure.
By \ref{C1}, $\partial\Omega_{a,\epsilon}$ is smooth. Since
\(u_{a,\epsilon}>0\) in \(\Omega_{a,\epsilon}\) and
\(u_{a,\epsilon}=0\) on \(\Sigma_{a,\epsilon}\), the Hopf
boundary lemma yields $\partial_\nu u_{a,\epsilon}<0$ on $\Sigma_{a,\epsilon}$. 
Consequently, 
\begin{equation*}
\sigma_{a,\epsilon}
:=
\chi_{T_{a,\epsilon}}\,dx
-
\bigl(\partial_\nu u_{a,\epsilon}\bigr)
\mathcal H^1\lfloor_{\Sigma_{a,\epsilon}}
\end{equation*}
is a nonnegative Radon measure, and
\begin{equation}\label{support-correct}
\operatorname{supp}\sigma_{a,\epsilon}
\subset\overline{T_{a,\epsilon}}.
\end{equation}
The divergence theorem gives
\[
        \int_{\partial\Omega_{a,\epsilon}}\partial_\nu u_{a,\epsilon}\dd\mathcal{H}^1
        =\int_{\Omega_{a,\epsilon}}\Delta u_{a,\epsilon}\dd x
        =-|\Omega_{a,\epsilon}|.
\]
Together with $\partial_\nu u_{a,\epsilon}\le0$ on the outer boundary, this gives
\[
        \sigma_{a,\epsilon}(B)
        =|T_{a,\epsilon}|-\int_{\Sigma_{a,\epsilon}}\partial_\nu u_{a,\epsilon}\dd \mathcal{H}^1
        \le |T_{a,\epsilon}|+|\Omega_{a,\epsilon}|=|B|.
\]
Thus the family $\{\sigma_{a,\epsilon}\}$ is uniformly bounded. By the compactness theorem for Radon measures
(see \cite{evans2015measure}), every sequence
$\epsilon_j\to0$ admits a subsequence, still denoted by
$\epsilon_j$, and there exists a nonnegative Radon measure $\sigma_a$ such that
\[
\sigma_{a,\epsilon_j}\rightharpoonup\sigma_a
\quad\text{as }j\to\infty.
\]
By \eqref{support-correct}  and  $d_{{H}}\!\left(\overline{{T_{a,\epsilon}}},\overline{S_a} \right) \to 0$  in \ref{C3}, we obtain
\begin{equation}\label{sigma-support}
        \supp\sigma_a\subset \overline{S_a}.
\end{equation}
Moreover, by Lemma \ref{smooth-convergence}, we have $H_{a,\epsilon}\to H_a$ strongly in $H^1_0(B)$.
Hence, for every $\varphi\in C_c^\infty(B)$,
\[
\int_B\varphi\,d\sigma_{a,\epsilon_j}
=
\left\langle-\Delta H_{a,\epsilon_j},\varphi\right\rangle
=
\int_B\nabla H_{a,\epsilon_j}\cdot\nabla\varphi\,dx
\to
\int_B\nabla H_a\cdot\nabla\varphi\,dx.
\]
Therefore,
\begin{equation}\label{sigma-representation}
-\Delta H_a=\sigma_a
\quad\text{in }\mathcal D'(B),
\end{equation} where $\mathcal D'(B)$ is the dual space of $C_c^\infty(B)$.
Since $-\Delta H_a=\sigma_a$ in $\mathcal D'(B)$, the measure
$\sigma_a$ is uniquely determined by $H_a$. Together with the uniform boundedness of
$\{\sigma_{a,\epsilon}\}$, this yields
\[
\sigma_{a,\epsilon}{\rightharpoonup}\sigma_a
\quad\text{as }\epsilon\to0.
\]
Let $G_B$ be the Green function of the unit disk,
\begin{equation*}
        G_B(x,\xi)=\frac1{2\pi}\log\frac{|\xi|\,|x-\xi^*|}{|x-\xi|},
        \qquad \xi^*:=\frac{\xi}{|\xi|^2}.
\end{equation*}
Combining $H_a\in H^1_0(B)$,  \eqref{sigma-support}, and \eqref{sigma-representation}, the Green representation
formula (see \cite{Ponce}) yields 
\begin{equation}\label{green-rep-ha}
        H_a(x)=\int_{\overline{S_a}}G_B(x,\xi)\dd\sigma_a(\xi) \quad \text{for } x\in D_a.
\end{equation}
The identity holds pointwise in $D_a$, since both sides are harmonic
and continuous away from $S_a$.
For $x=(0,y)$ and $\xi=(t,0)$ with $a\le |t|<1$, a direct computation gives
\begin{equation*}
        G_B\big((0,y),(t,0)\big)
        =\frac1{4\pi}\log\frac{1+t^2y^2}{t^2+y^2},
\end{equation*}
and hence
\begin{equation*}
        \partial_yG_B\big((0,y),(t,0)\big)=-yK(t,y),
\end{equation*}
where
\begin{equation*}
        K(t,y):=\frac{1-t^4}{2\pi(1+t^2y^2)(t^2+y^2)}.
\end{equation*}
 By $(1+t^2y^2)(t^2+y^2)\ge t^2$, we have
\begin{equation}\label{K-monotone}
        K(t,y)\le K(t,0)=\frac{1-t^4}{2\pi t^2} \quad \text{for all }y\in(0,1).
\end{equation}
Since \(\displaystyle \operatorname{dist}\big((0,y),S_a\big)>0\) for every \(y\in(0,1)\), differentiation of \eqref{green-rep-ha} with respect to \(y\) is justified. Therefore,
\[
        \partial_yH_a(0,y)=-y\int_{S_a}K(t,y)\dd\sigma_a(t).
\]
Combining $u_a=U-H_a$ and $\partial_yU(0,y)=-y/2$, we obtain
\begin{equation}\label{axis-derivative-formula}
        \frac1y\partial_yu_a(0,y)
        =-\frac12+\int_{S_a}K(t,y)\dd\sigma_a(t).
\end{equation}
Because $u_a$ is smooth near the origin and even in $y$,
\[
    \lim_{y\downarrow0}\frac{\partial_yu_a(0,y)}{y}
    =\partial_{yy}u_a(0)=\mu(a).
\]
Moreover, for each fixed $a>0$,
\[
    0\le K(t,y)\le K(t,0)\le \frac{1}{2\pi a^2},
    \quad t\in S_a \text{ and } y\in(0,1).
\]
Since $\sigma_a$ is finite, the dominated convergence theorem applied
to \eqref{axis-derivative-formula} gives
\[
    \mu(a)
    =-\frac12+\int_{S_a}K(t,0)\,\dd\sigma_a(t).
\]
Consequently, \eqref{axis-derivative-formula} and \eqref{K-monotone}
imply
\begin{equation}\label{axis-derivative-bound}
    \frac1y\partial_yu_a(0,y)\le \mu(a)
    \quad \text{for all } y\in(0,1).
\end{equation} Hence, if $\mu(a)<0$, \eqref{axis-derivative-bound} gives
\[
        \partial_yu_a(0,y)\le y\mu(a)<0
        \quad \text{for all }y\in(0,1).
\]
This proves the strict monotonicity on the positive  $y$-axis, and the negative side follows from the even symmetry in $y$.
\end{proof}

\begin{Prop}\label{origin-global-max-slit}
If $a\in(a^*,1)$, the origin is the unique global maximizer of $u_a$ in $D_a$.
\end{Prop}

\begin{proof}
By \eqref{critical-lambda-zero},  $\Lambda(a)<0$ for all $a\in(a^*,1)$. Hence $D^2u_a(0)$ is negative definite, and  in particular    $\mu(a)<0$.  Hence Lemma \ref{green-monotonicity} implies $u_a(0,y)<u_a(0,0)$ for $0<|y|<1$.  By Lemma \ref{lem:horizontal-monotonicity},  every global maximizer of $u_a$ must lie on the  $y$-axis.  Therefore the origin is the unique global maximizer of $u_a$ in $D_a$.
\end{proof}

Let $u_{a,\eps}$ be the solution of \eqref{fc1.1}. We next transfer the uniqueness of the global maximizer from the slit domain $D_a$ to the smooth approximating 
domains $\Omega_{a,\epsilon}$.

\begin{Prop}\label{fixed-a-stability}
Let $a\in(a^*,1)$. Then there exists $\epsilon(a)>0$ such that,
for every $0<\epsilon<\epsilon(a)$, the origin is the strict global
maximizer of $u_{a,\epsilon}$ in $\Omega_{a,\epsilon}$. Moreover,
\begin{equation}\label{D^2}
D^2u_{a,\epsilon}(0)
\longrightarrow
D^2u_a(0)
\quad\text{as }\epsilon\to 0.
\end{equation}
\end{Prop}

\begin{proof}
By Proposition \ref{origin-global-max-slit}, the origin  is the unique global maximizer of $u_a$ and  $D^2u_a(0)$ is negative definite for any fixed $a \in (a^*,1)$. Thus there exist constants $r,c>0$ such that 
$$B_r(0)\subset \subset D_a, \text{ and } D^2u_a(z)\le -cI_2 \quad \text{for all } z\in B_r(0).$$
By Lemma~\ref{smooth-convergence}, there exists
$\epsilon_1(a)>0$ such that, for every
$0<\epsilon<\epsilon_1(a)$,
$$B_r(0)\subset \subset \Omega_{a,\epsilon},  \text{ and } \displaystyle D^2u_{a,\epsilon}(z)\le -\frac c2 I_2 \quad \text{for all } z\in B_r(0).$$
Thus $u_{a,\epsilon}$ is strictly concave in $B_r(0)$.
Since $\Omega_{a,\epsilon}$ is symmetric with respect to both
coordinate axes, we have
$\nabla u_{a,\epsilon}(0)=0$,
and hence the origin is the unique maximizer of $u_{a,\epsilon}$
in $B_r(0)$.
\vskip0.1cm
It remains to exclude maximizers outside $B_r(0)$. Since
$u_a\in C(\overline{D_a})$ and the origin is its strict global
maximizer, we have
\[
\gamma_a
:=
u_a(0)
-
\max_{\overline{D_a}\setminus B_r(0)}u_a
>0.
\]
By Lemma~\ref{smooth-convergence},
$u_{a,\epsilon}(0)\to u_a(0)$ as $\epsilon\to0$. Hence there exists
$\epsilon_2(a)>0$ such that
\[
u_{a,\epsilon}(0)
>
u_a(0)-\frac{\gamma_a}{2}
\quad
\text{for every }0<\epsilon<\epsilon_2(a).
\]
On the other hand, Lemma~\ref{Lem2.2} gives
$u_{a,\epsilon}\le u_a$ in $\Omega_{a,\epsilon}$. Therefore, for
$z\in\Omega_{a,\epsilon}\setminus B_r(0)$,
\[
u_{a,\epsilon}(z)
\le u_a(z)
\le u_a(0)-\gamma_a
<
u_{a,\epsilon}(0).
\]
Set
\[
\epsilon(a)
:=
\min\big\{\epsilon_0(a),\epsilon_1(a),\epsilon_2(a)\big\},
\]
where $\epsilon_0(a)$ is  as in \ref{C4}.
Then, for every $0<\epsilon<\epsilon(a)$, the preceding estimates
show that the origin is the unique global maximizer of
$u_{a,\epsilon}$ in $\Omega_{a,\epsilon}$.
Finally, \eqref{D^2} follows from the local $C^2$ convergence in
Lemma~\ref{smooth-convergence}.
\end{proof}

\begin{proof}[\bf{Proof of  Theorem \ref{Thm1.1}}] 
 
By Proposition \ref{fixed-a-stability},  for any fixed $a\in ({a}^{*} ,1)$, there exists $ \epsilon(a)$ such that for all $\displaystyle \epsilon \in \big(0,\epsilon(a)\big)$, the origin  is a unique global maximizer of $u_{a,\epsilon}$ in $\Omega_{a,\epsilon}$, and  
$\displaystyle D^2u_{a,\epsilon}(0)\to D^2u_a(0)$.
Therefore, 
$$\displaystyle \lim_{\epsilon\to 0}\Lambda_\epsilon(a)
=
\Lambda(a), \quad
\text{where } \Lambda_\epsilon(a)
:=
\lambda_{\max}\bigl(D^2u_{a,\epsilon}(0)\bigr).$$
Since $a\in ({a}^{*} ,1)$ and $\displaystyle \epsilon \in \big(0,\epsilon(a)\big)$, Proposition \ref{fixed-a-stability} gives $\Lambda_\epsilon(a)\le0$.   On the other hand, continuity and \eqref{critical-lambda-zero} yield  $\displaystyle \lim_{a\downarrow a^*}\Lambda(a)=0$.
Hence,
\[
\lim_{a\downarrow a^*}
\lim_{\epsilon \to 0}
\Lambda_\epsilon(a)
=
\lim_{a\downarrow a^*}
\Lambda(a)
=
0.
\]
Set  $z_*:=\left(0,1/2\right)$. Then $\overline{B_{1/4}(z_*)} \subset \subset D_a$ for any $a\in 
(0,1)$.  Indeed, if \(z=(x,y)\in \overline{B_{1/4}(z_*)}\), then
$
y\ge 1/4>0
$. Since the slit \(S_a\) is contained in the horizontal axis
\(\{y=0\}\), it follows that $z\notin S_a$. Moreover, we have
\[
|z|
\le |z_*|+|z-z_*|
\le \frac12+\frac14
=\frac34<1,
\] and therefore \(z\in B_1(0)\). Thus $\overline{B_{1/4}(z_*)} \subset \subset D_a$.
For each fixed $a \in (0,1)$, by \ref{C4}, the family $\{\Omega_{a,\epsilon}\}_{\epsilon>0}$ is increasing as
$\epsilon\to0$ and exhausts $D_a$. Since $\overline{B_{1/4}(z_*)}$ is compact,
there exists $\bar\epsilon(a)>0$ such that $\overline{B_{1/4}(z_*)}\subset\Omega_{a,\epsilon}$ holds for any $0<\epsilon<\bar\epsilon(a)$.
 We can assume that $0<\epsilon_0(a)<\bar\epsilon(a)$. Hence $\operatorname{inrad}(\Omega_{a,\epsilon})
\ge 1/4$ holds for any $0<\epsilon<\epsilon_0(a)$. Note that $\operatorname{diam}(\Omega_{a,\epsilon})
\le
\operatorname{diam}(B)
=2$. 
Therefore \eqref{inrad} holds.
\end{proof}

\appendix         
\section{Construction of domains satisfying \texorpdfstring{\eqref{quyuomiga}}{(quyuomiga)} and \texorpdfstring{\ref{C1}--\ref{C6}}{(C1)--(C6)}}\label{A}
\setcounter{equation}{0}
\renewcommand{\theequation}{A.\arabic{equation}}
\setcounter{equation}{0}
We  construct a cutoff function and then use it to define domains satisfying all assumptions.
 \begin{Lem}
\label{adapted-cutoff}
For every $\delta>0$, there exists an even function
$\phi_\delta\in C_c^\infty(\mathbb R)$ with
$\operatorname{supp}\phi_\delta\subset\subset(-1,1)$, such that $0\leq \phi_\delta\leq 1$, $\phi_\delta(0)=1$, 
\begin{equation}
\label{eq:cutoff-differential}
\phi_\delta'(s)\leq 0 \quad \text{and} \quad
\phi_\delta(s)-s\phi_\delta'(s)
\leq 1+\delta \quad\text{for }0\leq s \leq 1.
\end{equation}
\end{Lem}

\begin{proof}
Choose $ \tau\in(0,\frac{1}{4})$ sufficiently small so that
$ \frac{2\tau}{1-2\tau}\leq \delta$. Choose $\displaystyle w_\tau\in C^\infty\big([0,1]\big)$ such that $0\leq w_\tau\leq1$ and
\begin{equation}\label{wtau}
    w_\tau(s)=\frac{s}{\tau}
\; \;\left(0\leq s\leq\frac{\tau}{4}\right), \quad
w_\tau(s)=1\; \;(\tau\leq s\leq1-\tau),
\quad
w_\tau(s)=0
\; \; \left(1-\frac{\tau}{2}\leq s\leq1\right).
\end{equation}
Let $\displaystyle C_\tau:=\int_0^1 w_\tau(r)\,dr.$
Since $w_\tau=1$ on $[\tau,1-\tau]$, we have $C_\tau\geq1-2\tau>0$.
For $s\in[0,1]$, define $$\displaystyle \phi_\delta(s)
:=
\frac{1}{C_\tau}\int_s^1 w_\tau(r)\,dr.$$
Extend $\phi_\delta$ evenly to $(-1,1)$ and   extend it by zero
outside $(-1,1)$.
For $0\leq s\leq \tau/4$, \eqref{wtau} gives
$\phi_\delta(s)
=
1-\frac{s^2}{2\tau C_\tau}$.
Hence the even extension is smooth at the origin. Since $\phi_\delta$ vanishes identically in neighborhoods of $\pm1$, its extension by zero belongs to $C_c^\infty(\mathbb R)$ and $\operatorname{supp}\phi_\delta\subset\subset(-1,1)$.
By construction, $\phi_\delta(0)=1$ and $0\leq\phi_\delta\leq1$. Moreover, $\phi_\delta'(s)=-\frac{w_\tau(s)}{C_\tau}\leq0$.
It remains to prove \eqref{eq:cutoff-differential}.  For
$0\leq s\leq1$,
\begin{equation}\label{eq:eta-computation}
  \displaystyle  \phi_\delta(s)-s\phi_\delta'(s)=
\frac{1}{C_\tau}
\left(
\int_s^1 w_\tau(r)\,dr+s w_\tau(s)
\right)
=
1+
\frac{
s w_\tau(s)-\int_0^s w_\tau(r)\,dr
}{C_\tau}.
\end{equation}
We claim that
\begin{equation}\label{eq:w-error}
s w_\tau(s)-\int_0^s w_\tau(r)\,dr
\leq2\tau,
\quad 0\leq s\leq1.
\end{equation}
Indeed, if $0\leq s\leq\tau$, the left-hand side is at
most $s\leq\tau$. If $\tau\leq s\leq1-\tau$, then
$w_\tau(s)=1$ and 
\[
s w_\tau(s)-\int_0^s w_\tau(r)\,dr
\leq
s-\int_\tau^s1\,dr
=\tau.
\]
If $1-\tau\leq s\leq1$, 
\[
s w_\tau(s)-\int_0^s w_\tau(r)\,dr
\leq
s-\int_\tau^{1-\tau}1\,dr
\leq2\tau.
\]
Combining
\eqref{eq:eta-computation}, \eqref{eq:w-error}, and
$C_\tau\geq1-2\tau>0$, we obtain
\[
\phi_\delta(s)-s\phi_\delta'(s)
\leq
1+\frac{2\tau}{C_\tau}
\leq
1+\frac{2\tau}{1-2\tau}
\leq1+\delta.
\]
\end{proof}

Take $\delta=\frac{a}{4}$ in Lemma \ref{adapted-cutoff}  and denote the corresponding cutoff function by $\phi_a$. Then $\displaystyle \phi_a\in C_c^\infty\big((-1,1)\big)$, $\phi_a(-s)=\phi_a(s)$, $0\leq\phi_a\leq1$,  $\phi_a(0)=1$ and 
\begin{equation*}
\phi_a'(s)\leq0, \quad
\phi_a(s)-s\phi_a'(s)
\leq1+\frac a4,
\quad \text{for all } s \in [0,1].
\end{equation*}
Define $b(y):=\sqrt{1-y^2}$ for $|y|<1$ and set $ x_a:=\frac{1+a}{2}$. Choose $\eps_0(a)>0$ so that
\begin{equation}
\label{eps0-choice}
0<\eps_0(a)
<
\min\left\{
 a/4,\,
{\sqrt{1-x_a^2}}/4
\right\}.
\end{equation}
For $0<\eps<\eps_0(a)$, let
\begin{equation}
\label{eq:rho-definition}
\rho_{a,\eps}(y)
:=
b(y)
-
(1-a+\eps)
\phi_a\left(\frac{y}{\eps}\right),
\quad |y|<1,
\end{equation}
where $\phi_a$ is extended by zero outside $(-1,1)$.
Define the closed right-hand slot by
\begin{equation*}
\overline{T^+_{a,\eps}}
:=
\left\{
(x,y)\in\overline B:
|y|\leq\eps,\quad
\rho_{a,\eps}(y)\leq x\leq b(y)
\right\},
\end{equation*}
set $\displaystyle \overline{T^-_{a,\eps}}
:=
\big\{(-x,y):(x,y)\in \overline{T^+_{a,\eps}}\big\}$, $\overline{T_{a,\eps}}
:=
\overline{T^+_{a,\eps}}\cup \overline{T^-_{a,\eps}}$, and define $$\Omega_{a,\eps}
:=
B\setminus \overline{T_{a,\eps}}.$$
 If $|y|\geq\eps$, then $ \phi_a\left(\frac{y}{\eps}\right)=0$
and  $\rho_{a,\eps}(y)=b(y)>0$.
If $|y|\leq\eps$, then $0\leq\phi_a\leq1$ and
$\sqrt{1-y^2}\geq1-y^2$. Thus 
\begin{equation}\label{A10}
    \rho_{a,\eps}(y)\geq
1-\eps^2-(1-a+\eps)
=
a-\eps-\eps^2.
\end{equation}
Since $\displaystyle \epsilon \in \big(0,\eps_0(a)\big)$,  \eqref{eps0-choice} gives $\eps<a/4\le 1/4$ and hence $\eps^2\leq{\eps}/{4}< a/{16}$. Consequently, \eqref{A10} yields
$\rho_{a,\eps}(y)
\geq
\frac{11a}{16}>0$. Thus $\rho_{a,\eps}$ is positive on $(-1,1)$. By the definition of $\Omega_{a,\eps}$, we have
\begin{equation}
\label{eq:Omega-cross-section}
\Omega_{a,\eps}
=
\left\{
(x,y)\in\R^2:
|y|<1,\quad
-\rho_{a,\eps}(y)<x<\rho_{a,\eps}(y)
\right\}.
\end{equation}
\noindent \textbf{Verification of \ref{C1}, \ref{C2}, and \ref{C5}:} 
If $\eps\leq y<1$,
then $ yb'(y)-b(y)
=
-\frac{1}{\sqrt{1-y^2}}<0$.
If  $0<y<\eps$, then $s:=\frac{y}{\eps}\in(0,1)$.
Hence \eqref{eq:rho-definition} gives \begin{equation}\label{A11}
    y\rho_{a,\eps}'(y)-\rho_{a,\eps}(y)
=
-\frac{1}{\sqrt{1-y^2}}
+
(1-a+\eps)
\bigl(\phi_a(s)-s\phi_a'(s)\bigr).
\end{equation}
By  \eqref{eps0-choice}, we have $1-a+\eps
\leq 1-\frac{3a}{4}$.
Therefore, 
\begin{equation}\label{A12}
(1-a+\eps)
\bigl(\phi_a(s)-s\phi_a'(s)\bigr)
\leq
\left(1-\frac{3a}{4}\right)
\left(1+\frac a4\right)
=
1-\frac a2-\frac{3a^2}{16}
<1.
\end{equation}
Combining \eqref{A11} and \eqref{A12} yields
\[y\rho_{a,\eps}'(y)-\rho_{a,\eps}(y)
\leq
-\frac a2-\frac{3a^2}{16}<0.\]
Hence, for $0<t<1$ and $0<y<1$, $\rho_{a,\eps}(ty)
>
t\rho_{a,\eps}(y)$.
By evenness, $\rho_{a,\eps}(ty)
>
t\rho_{a,\eps}(y)$ also holds with $y$ replaced by
$|y|$. Therefore,
if $(x,y)\in\overline{\Omega}_{a,\eps}$ and  $0<t<1$,
then 
\[
\displaystyle |tx|
\leq
t\rho_{a,\eps}\big(|y|\big)
<
\rho_{a,\eps}\big(t|y|\big) \quad \text{for} ~0<|y|<1, \quad \text{and thus}~ (tx,ty)\in\Omega_{a,\eps}.\]
The cases $y=0$ and $(x,y)=(0,\pm1)$ are immediate. Thus,  for any $0\leq t<1$, $t\overline{\Omega}_{a,\eps}\subset\Omega_{a,\eps}$, 
which proves \ref{C5}. In particular,
$\Omega_{a,\eps}$ is contractible and hence simply connected. 

We next verify that $\Omega_{a,\epsilon}$ is nonconvex. Recall that $ x_a=\frac{1+a}{2}$.
\eqref{eps0-choice} gives $ 4\eps^2
<
\frac{1-x_a^2}{4}$
and hence $x_a^2+4\eps^2<1$.
Thus the two points $P_\pm:=(x_a,\pm2\eps)$
belong to $B$. Since $|2\eps|>\eps$, the cutoff function vanishes at
$y=\pm2\eps$, and therefore $\rho_{a,\eps}(\pm2\eps)=b(\pm2\eps)>x_a$.
It follows that $P_+,P_-\in\Omega_{a,\eps}$.
On the other hand, we have $\rho_{a,\eps}(0)=a-\eps<x_a$, 
so their midpoint $ \frac{P_++P_-}{2}=(x_a,0)$
does not belong to $\Omega_{a,\eps}$. Hence
$\Omega_{a,\eps}$ is nonconvex. Since $\operatorname{supp}\phi_a\subset\subset(-1,1)$, the perturbation term $\phi_a\left(\frac{y}{\eps}\right)$
vanishes in neighborhoods of $y=\pm\eps$. Hence
$\rho_{a,\eps}$ agrees with
$b(y)=\sqrt{1-y^2}$ near these points and therefore belongs to
$C^\infty$ on $(-1,1)$.
The boundary consists of the two curves
\[
\Gamma^\pm_{a,\eps}
:=
\left\{
\bigl(\pm\rho_{a,\eps}(y),y\bigr):
-1<y<1
\right\}.
\]
By \eqref{A10}, the two curves do not intersect in
$-1<y<1$.  Near the points $(0,\pm1)$, the cutoff term vanishes and
$\partial\Omega_{a,\epsilon}$ agrees locally with the unit circle. Hence,
$\partial\Omega_{a,\eps}$ is a $C^\infty$ embedded Jordan curve
and $\Omega_{a,\eps}$ is a bounded $C^\infty$ domain. The symmetry of $\Omega_{a,\eps}$ with respect to both coordinate axes follows immediately from the evenness of $\rho_{a,\eps}$ and the definition of $\Omega_{a,\eps}$.
This proves \ref{C1}. Furthermore, since $\rho_{a,\eps}(0)
=
1-(1-a+\eps)
=
a-\eps<a$,
we have $S_a\subset \overline{T_{a,\eps}}$, proving \ref{C2}.
\vskip0.15cm
\noindent \textbf{Verification of \ref{C3}:} By \ref{C2}, we have $\overline{S_a}\subset \overline{T_{a,\eps}}$,
and therefore $\displaystyle \sup_{p\in\overline{S_a}}
    \dist\bigl(p,\overline{T_{a,\eps}}\bigr)=0$.
It remains to estimate the distance in the opposite direction.
Let $z=(x,y)\in\overline{T^+_{a,\eps}}$. Then
$|y|\le \eps$ and $x\ge \rho_{a,\eps}(y)$, so that \eqref{A10} yields    $x\ge a-\eps-\eps^2$.
If $x\ge a$, then $(x,0)\in\overline{S_a}$, and hence
   $\displaystyle  \dist\bigl(z,\overline{S_a}\bigr)
    \le |y|
    \le \eps$.
If $x<a$, then    $0<a-x\le \eps+\eps^2$.
Since $(a,0)\in\overline{S_a}$, 
\[
    \dist\bigl(z,\overline{S_a}\bigr)
    \le |z-(a,0)|
    \le \sqrt{(\eps+\eps^2)^2+\eps^2}
    \le 2\eps,
\]
where we used $\eps<1/4$. By symmetry, the same estimate holds for
$z\in\overline{T^-_{a,\eps}}$. Consequently,
$
    \sup\limits_{z\in\overline{T_{a,\eps}}}
    \dist\bigl(z,\overline{S_a}\bigr)
    \le 2\eps.
$
Thus
$d_H\bigl(\overline{T_{a,\eps}},\overline{S_a}\bigr)
    \le 2\eps\to0.$
\vskip0.15cm
\noindent \textbf{Verification of \ref{C4}:}  Let $0<\eps_1<\eps_2<\eps_0(a)$. Fix $y\in(-1,1)$ and set $$s_i:=\frac{|y|}{\eps_i}, \quad  A_i:=1-a+\eps_i, \quad  \text{for}~~ i=1,2.$$
Since $\phi_a$ is nonincreasing on $[0,\infty)$, 
$\phi_a(s_2)\geq\phi_a(s_1)$.
Moreover, $A_2>A_1>0$ yields
\begin{equation*}
    A_2\phi_a(s_2)-A_1\phi_a(s_1)
=
(A_2-A_1)\phi_a(s_2)
+
A_1\bigl(\phi_a(s_2)-\phi_a(s_1)\bigr)
\geq0,
\end{equation*}
thus $\rho_{a,\eps_2}(y)
\leq
\rho_{a,\eps_1}(y)$ for all $|y|<1$. \eqref{eq:Omega-cross-section} gives $\Omega_{a,\eps_2}
\subset
\Omega_{a,\eps_1}$, and therefore
$\displaystyle \bigcup_{0<\eps<\eps_0(a)}
\Omega_{a,\eps}
\subset D_a.$
Let $(x,y)\in D_a$.
If $y\neq0$, choose $\displaystyle \epsilon \in \big(0,\min\{\eps_0(a),|y|\}\big)$.
Then $|y|/\eps>1$, and  $\phi_a\left(\frac{y}{\eps}\right)=0$.
Therefore $\rho_{a,\eps}(y)=b(y)$.
Since $(x,y)\in B$, $|x|<b(y)$, and thus $(x,y)\in\Omega_{a,\eps}$.
If $y=0$, then $(x,0)\in D_a$ implies $|x|<a$. Choose $\displaystyle \epsilon \in \big(0,\min\{\eps_0(a),a-|x|\}\big)$.
Then \[|x|<a-\eps=\rho_{a,\eps}(0) \quad \text{and} \quad
 (x,0)\in\Omega_{a,\eps}.\]
Thus $\displaystyle D_a
\subset
\bigcup_{0<\eps<\eps_0(a)}
\Omega_{a,\eps}$. Together with the opposite inclusion proved above, this proves 
\ref{C4}.
\vskip0.15cm
\noindent \textbf{Verification of \ref{C6}:}  For each fixed $y\in(-1,1)$, the horizontal
section of $\Omega_{a,\epsilon}$ is the interval
\[
\bigl\{x\in\mathbb R:(x,y)\in\Omega_{a,\epsilon}\bigr\}
=
\bigl(-\rho_{a,\epsilon}(y),\rho_{a,\epsilon}(y)\bigr).
\]
Hence, if $(x_1,y),(x_2,y)\in\Omega_{a,\epsilon}$ and $t\in[0,1]$,
then
\[
-\rho_{a,\epsilon}(y)
<
(1-t)x_1+tx_2
<
\rho_{a,\epsilon}(y).
\]
Therefore, $\bigl((1-t)x_1+tx_2,y\bigr)\in\Omega_{a,\epsilon}$,
and $\Omega_{a,\epsilon}$ is convex in the $x$-direction. This proves
\ref{C6}.

 \vskip0.2cm

\section{Proof of  Lemma \ref{capacity} }\label{B}
\setcounter{equation}{0}
\renewcommand{\theequation}{B.\arabic{equation}}
\setcounter{equation}{0}

\begin{proof}[\bf{Proof of  Lemma \ref{capacity}}]
We use the planar Wiener criterion formulated in terms of relative variational 2-capacity. For the relevant capacity theory and the Wiener criterion, see \cite{Mazya2015}. If $K \subset \subset G$ is compact, set
\[
\operatorname{Cap}_2(K, G) := \inf\left\{ \int_G |\nabla \psi|^2 \, dx : \psi \in C_c^{\infty}(G), \, \psi \geq 1 \text{ in a neighborhood of } K \right\}.
\] In two dimensions, the relative 2-capacity is scale invariant:
\begin{equation}\label{36}
    \operatorname{Cap}_2(q + rK, q + rG) = \operatorname{Cap}_2(K, G).
\end{equation}
Every point of $ \partial  B$, including $(\pm1,0)$, satisfies an exterior ball condition and 
is therefore regular. It remains to consider
\[
q \in S_a = \{(x, 0) : a \leq |x| < 1\}.
\]
Fix such a point. There exist $r_q > 0$, a unit vector $e_q$ tangent to the slit, and a number $\kappa \in (0, 1/2]$ such that, for every $0 < r < r_q$, we have
\begin{equation}\label{37}
    L_{q, r} := \{ q + t e_q : 0 \leq t \leq \kappa r \} \subset (\mathbb{R}^2 \setminus D_a) \cap \overline{B_r(q)}.
\end{equation}
At either inner slit endpoint $q = (\pm a, 0)$, choose the segment in the outward direction along the
slit. At an interior slit point, either tangential direction is admissible.
We next show that a nondegenerate line segment has positive relative 2-capacity. Let
\[
L := \{ t e_1 : 0 \leq t \leq \kappa \} \subset \subset B_2(0).
\]
Choose a fixed Lipschitz rectangle $Q \subset \subset B_2(0)$ lying on one side of $L$ and having $L$ as a subsegment of one of its sides. The trace inequality on $Q$ and  Poincar\'e's inequality on $B_2$ give 
\begin{equation}\label{38}
    \int_L |\psi|^2 \, d\mathcal{H}^1 \leq C_Q \|\psi\|_{H^1(Q)}^2 \leq C \int_{B_2(0)} |\nabla \psi|^2 \, dx \quad \text{for any } \psi \in C_c^{\infty}(B_2(0)),
\end{equation}
and see \cite{grisvard2011elliptic} for the standard trace results. If $\psi \geq 1$ in a neighborhood of $L$, then the left-hand side of \eqref{38} is at least $\mathcal{H}^1(L) = \kappa$. Consequently,
\begin{equation*}
    c_0 := \operatorname{Cap}_2\big(L, B_2(0)\big) \geq \frac{\kappa}{C} > 0. 
\end{equation*}
By monotonicity of capacity, \eqref{37}, rotation invariance, and \eqref{36},
\begin{equation}\label{40}    
\operatorname{Cap}_2\big( (\mathbb{R}^2 \setminus D_a\big) \cap \overline{B_r(q)}, B_{2r}(q) \big) \geq \operatorname{Cap}_2\big(L_{q, r}, B_{2r}(q)\big) 
= \operatorname{Cap}_2\big(L, B_2(0)\big) = c_0
\end{equation}
for every $0 < r < r_q$.
For completeness, we recall the following normalized discrete form of the planar Wiener
criterion (see \cite{Mazya2015}). Set
$E_j := (\mathbb{R}^2 \setminus D_a) \cap \overline{B_{2^{-j}}(q)}$.
A boundary point $q$ is regular if and only if
\begin{equation}\label{41}
    \sum_{j = j_0}^{\infty} \frac{\operatorname{Cap}_2\big(E_j, B_{2^{1-j}}(q)\big)}{\operatorname{Cap}_2\big(\overline{B_{2^{-j}}(q)}, B_{2^{1-j}}(q)\big)} = +\infty.
\end{equation}
The denominator is independent of $j$. Indeed, the radial capacitary potential of the condenser $\displaystyle \big(\overline{B_r(q)}, B_{2r}(q)\big)$ is
\[
\Psi_{q, r}(x) = \begin{cases}
1, & |x - q| \leq r, \\
\frac{\log\big(2r / |x - q|\big)}{\log 2}, & r < |x - q| < 2r, \\
0, & |x - q| \geq 2r,
\end{cases}
\]
and radial minimization gives
\begin{equation}\label{42}
    \operatorname{Cap}_2\left( \overline{B_r(q)}, B_{2r}(q) \right) = \int_{B_{2r}(q) \setminus B_r(q)} |\nabla \Psi_{q, r}|^2 \, dx = \frac{2\pi}{\log 2}.
\end{equation}
Choose $j_0$ so that $2^{-j_0} < r_q$. By \eqref{40} and \eqref{42}, every summand in \eqref{41} is bounded below by $c_0 \log 2 / (2\pi) > 0$. Hence the Wiener series diverges, and every slit point is regular. Thus every point of $\partial D_a$ is regular.
It remains to identify the boundary values of the torsion function. Let $h_a := U - u_a$, where $U$ is defined in \eqref{UU}.
Then $h_a$ is harmonic in $D_a$ and $h_a - U = -u_a \in H_0^1(D_a)$. 
Thus $h_a$ is the variational solution of the harmonic Dirichlet
problem in $D_a$ with boundary data $U|_{\partial D_a}$. By the
standard identification of the variational and Perron solutions for
continuous boundary data, $h_a$ coincides with the Perron solution
corresponding to $U|_{\partial D_a}$. Since every point of
$\partial D_a$ is regular for the Dirichlet problem, it follows that
\[
\lim_{D_a\ni x\to q}u_a(x)
=
\lim_{D_a\ni x\to q}\bigl(U(x)-h_a(x)\bigr)
=
U(q)-U(q)
=
0
\quad\text{for every }q\in\partial D_a.
\]
Since $u_a$ is smooth in $D_a$ by interior elliptic regularity, we
conclude that
\[
u_a\in C(\overline{D_a}),
\qquad
u_a=0\quad\text{on }\partial D_a.
\]
\end{proof}

\section*{Acknowledgments}
\vskip0.1cm
 P. Luo was supported by the National Key R\&D Program (No. 2023YFA1010002) and the National Natural Science Foundation of
China (Grant No. 12422106).
\vskip0.1cm
\noindent\textbf{Data availability statement:}  No data were generated or analyzed in this study.
\vskip0.1cm

\noindent\textbf{AI assistance statement:} The authors used OpenAI tools solely for language polishing and editorial assistance. The
authors independently verified all mathematical arguments, references, and conclusions
and take full responsibility for the content.

\end{document}